\documentclass[12pt,reqno]{amsart}

\usepackage{verbatim} \usepackage{marginnote} \usepackage{amsmath}
\usepackage{amsfonts} \usepackage{amssymb} \usepackage{amsthm}
\usepackage{graphicx} \usepackage{setspace}
\usepackage{enumerate} \usepackage{mathrsfs} \usepackage{empheq}
\usepackage{graphics} \usepackage{epsfig}
\usepackage[colorlinks=true, linkcolor=red, citecolor=blue]{hyperref}
\usepackage[left=1in,top=1in,right=1in, bottom=1in]{geometry}
\usepackage{graphicx, psfrag}
\usepackage{stackrel}
\usepackage{color}

 \DeclareMathOperator{\spt}{spt}

\DeclareMathOperator{\esssup}{ess\ sup}

\newcommand{\Lam}{\Lambda} 
\newcommand{\Om}{\Omega} \newcommand{\tht}{\theta}
\newcommand{\Tht}{\Theta} \newcommand{\al}{\alpha}
\newcommand{\eps}{\epsilon} \newcommand{\be}{\beta}
 \newcommand{\De}{\Delta}
\newcommand{\de}{\delta} \newcommand{\s}{\sigma}
 \newcommand{\gam}{\gamma}
\newcommand{\Gam}{\Gamma} \newcommand{\kap}{\kappa}
\newcommand{\lam}{\lambda} 
\newcommand{\Ph}{\Phi} 
\newcommand{\z}{\zeta}

\def\eric{\color[rgb]{0,0,0}}
\def\mystara{{\eric \mathcal{S}}}
\def\mystarb{{\eric \widetilde{\mathcal{S}}}}
\def\I{{\eric \mathcal{I}}}
\def\J{{\eric \mathcal{J}}}
\def\K{{\eric \mathcal{K}}}

\def\vincent{\color[rgb]{0,0,0}}

\newcommand{\ind}{{\eric\mathbb{I}}}

 \newcommand{\R}{\mathbb{R}}
 
\newcommand{\T}{\mathbb{T}} \newcommand{\Z}{\mathbb{Z}}
 \newcommand{\A}{\mathcal{A}}

\newcommand{\del}{\nabla} \newcommand{\bdy}{\partial}

 \newcommand{\Ri}{\mathcal{R}}

 \newcommand{\til}[1]{\widetilde{#1}}

 \newcommand{\smod}{\setminus}

 \newcommand{\sub}{\subset}

\newcommand{\ls}{\lesssim}

\newcommand{\goesto}{\rightarrow}

\newcommand{\imb}{\hookrightarrow}

\newcommand{\Hper}{H_{{per}}^\s}
\newcommand{\Hdper}{\dot{H}_{{per}}^\s}

\newcommand{\Sob}[2]{\lVert#1\rVert_{#2}}

\newcommand{\abs}[1]{\lvert#1\rvert}

\newcommand{\lb}{\langle} \newcommand{\rb}{\rangle}

\newcommand{\req}[1]{(\ref{#1})}

 \newcommand{\ek}{\eta^{(k)}}
\newcommand{\vk}{v^{(k)}}

\newcommand{\FLp}{F_{L^p}} \newcommand{\FHg}{F_{H^{-\gam/2}}}
 \newcommand{\TLt}{\Tht_{L^2}}
\newcommand{\TLp}{\Tht_{L^p}} \newcommand{\THs}{\Tht_{H^\s}}

\newtheorem{thm}{Theorem} 
 
\newtheorem{prop}{Proposition}[subsection]
\newtheorem{lem}{Lemma}[subsection]
\newtheorem{rmk}{Remark}[section]
\newtheorem{coro}[lem]{Corollary}

\newtheorem*{stand}{Standing Hypotheses}

\numberwithin{equation}{section}

\title[data assimilation using blurred-in-time observations]
	{Continuous data assimilation with
	blurred-in-time measurements
    of the surface quasi-geostrophic equation}

\author{Michael S. Jolly$^{1}$} \address{$^{1}$Department of
  Mathematics\\ Indiana University\\ Bloomington, IN 47405}
\email[M.S. Jolly]{msjolly@indiana.edu}

\author{Vincent R. Martinez$^{2}$} \address{$^{2}$Department of
  Mathematics and Statistics\\ CUNY-Hunter College\\ New York, NY 10065}
\email[V. R. Martinez]{vrmartinez@hunter.cuny.edu}

\author{Eric J. Olson$^{3}$} \address{$^{3}$Department of Mathematics
  and Statistics\\ University of Nevada-Reno\\ Reno, NV 89557}
\email[E.J. Olson]{ejolson@unr.edu}

\author{Edriss S. Titi$^{4, \dagger}$} \address{$^4$ Department of Mathematics
  \\ Texas A\& M University \\ College Station, TX
  77843-3368\\ Also:\\ Department of Computer Science and Applied
  Mathematics\\ Weizmann Institute of Science\\ Rehovot, 76100,
  Israel} \email[E.S. Titi]{titi@math.tamu.edu  $\dagger$}

\thanks{$\dagger$ Denotes corresponding author}

\begin{document}
\date{August 31, 2018}

\begin{abstract}
{An intrinsic property of almost any physical measuring device is
that it makes observations which are slightly blurred in time.
We consider a nudging-based approach for data assimilation that constructs an
approximate solution based on a feedback control mechanism that is designed to
account for observations that have been blurred by a moving
time average.
Analysis of this nudging model in the context of the subcritical
surface quasi-geostrophic equation shows,
provided the time-averaging window is sufficiently small and
the resolution of the observations sufficiently fine,
that the approximating solution converges exponentially fast
to the observed solution over time.
In particular, we demonstrate that observational data with a small blur in time
possess no significant obstructions to data assimilation provided that
the nudging properly takes the time averaging into account.
Two key ingredients in our analysis are additional boundedness properties for the relevant interpolant observation operators and a non-local Gronwall
inequality.
}
\end{abstract}



\maketitle

\centerline{\it Dedicated to Professor Andrew Majda on the occasion of his 70th birthday}

\section{Introduction}

The surface quasi-geostrophic (SQG) equation models the dynamics
of the potential temperature on the two-dimensional horizontal
boundaries of the three-dimensional quasi-geostrophic equations,
which, in turn, are approximations to the shallow water equations
in the limit of small Rossby number where the inertial forces are
an order of magnitude smaller than the Coriolis and pressure forces.
This is the regime of strong rotation, where
the time scales associated with atmospheric flow over long distances
are much larger than the time scales associated with the Earth's
rotation (cf. \cite{pedlosky}).
{\vincent{The model of focus in our study of data assimilation
is the {subcritically} dissipative SQG equation subject to periodic boundary conditions over the fundamental domain $\T^2=[-\pi,\pi]^2$. In \textit{non-dimensionalized} variables, it is given by}}
\begin{align}\label{sqg}
		\begin{cases}
			\bdy_t\tht+\kap\Lam^\gam\tht+u\cdotp{\del}\tht=f,\quad
                        \\ u=\Ri^\perp\tht, \quad
                        \tht(x,0)=\tht_0(x),
		\end{cases}
	\end{align}
{\vincent{where $\Lam^\gam=(-\De)^{\gam/2}$ corresponds to the Fourier muliplier operator $|k|^\gam$,
$\Ri^\perp=(-R_2,R_1)$ is the perpendicular Riesz transform, where each $R_j$ corresponds to
$(-ik_j/|{\mathbf{k}}|)_{\mathbf{k}\in\Z^2\smod\{0\}}$, and the strength of dissipation satisfies $1<\gam\leq2$.}}
Note that $\gamma=1$ gives the
critical case while $0<\gamma<1$ gives the supercritical case.
The scalar function $\tht$ represents the surface
temperature or buoyancy of a fluid advected along the vector velocity
field $u$.
The parameter
$\kap$ is a fixed positive quantity,
which appears due to the phenomenon of Ekman pumping at the
surface.
{\vincent{Note, also, that if $\tht_0$ has zero mean over $\T^2$, then the property $\frac{1}{4\pi^2}\int_{\T^2}\tht(t)dx=0$ is propagated for all $t>0$, so long as $f$ has zero mean over $\T^2$ as well.}}


Since their
introduction into the mathematical community by Constantin, Majda, and
Tabak \cite{cmt}, the subcritical, critical and supercritical SQG equations
have been thoroughly studied.
Well-posedness and global regularity
in various function spaces has been resolved in all but the
supercritical case, (cf. \cite{caff:vass, const:vic, const:wu:qgwksol,
ctv, cz:vic, hdong, kis:naz, kis:naz:vol, resnick}), and also for certain inviscid regularizations (cf. \cite{khouider:titi}).
The long-time behavior in the subcritical and critical has been studied as well and in particular, a global attractor theory has been established for them
(cf. \cite{carrillo:ferreira, chesk:dai:qgattract, const:coti:vic,
  ctv, coti:zelati, ju:qgattract}).
These equations have been used to simulate the production of
fronts in geophysical flows and in spite of being a scalar model
in two dimensions, possess solutions that behave in ways that are
strikingly similar to fully three-dimensional flows.
Therefore, equations (\ref{sqg}) provide a physically-relevant
dynamical context in which to analyze the performance of our model for
data assimilation, that also supplies additional analytical difficulties
that requires us to further develop the theoretical foundations of our approach.

Given a geophysical equation that describes some aspect of
reality, the ability to predict the future using
this equation requires an initial condition that accurately
represents the current physical state.
Although weather data has been collected nearly continuously
in time since the 1960s, this data represents, at best, an
incomplete picture of the current state of the atmosphere.
Thus, rather than an exact initial condition, in practice one has
a time series of low-resolution observations.
Moreover, due to the nature of the measuring devices, the
data itself may contain noise as well as systematic errors.
Of particular interest to our present study is the fact
that nearly all physical instrumentation
produces measurements which are \textit{manifestly blurred in time}.
{\eric For example,
the heat capacity of a thermometer naturally averages
temperatures as they change over time while
the rotational inertial of
an anemometer similarly averages velocities.
Time averages in satellite images result from finite shutter speeds and
further averages result when satellite data is obtained by comparing
two subsequent images.}
Blocher \cite{blocher} shows both
analytically and computationally that noisy, blurred-in-time
observations of the $X$ variable can be used to synchronize
two copies of the {\vincent{three-dimensional Lorenz system of ordinary differential equations (ODEs)}} up to a factor of the
variance of the noise, {\eric see also \cite{blocher18}.}
As the analysis of the SQG equation is more complicated,
we do not consider noise or systematic
errors in this work, {\vincent{as this was studied in \cite{bessaih:olson:titi} and \cite{foias:mondaini:titi}}}, but instead focus solely on how to
assimilate data that has been subject to a moving time average.

The idea of finding the current
physical state by combining a time-series of partial
observations with knowledge about the dynamics dates back
to a 1969 paper of Charney, Halem, and Jastrow \cite{chj}.
Doing this optimally is the subject of data assimilation.
Data assimilation has received considerable attention
in both its theoretical development and practical use
for the prediction of the weather (cf. Kalnay \cite{kalnay}
and references therein).
The approach of
interest in this article computes an approximation using a
``auxiliary system" obtained by taking the original model,
which is assumed to coincide with the observations
in the absence of measurement error, and applying feedback control
{based on the observations}.
This feedback control serves to nudge the solution towards
the unknown but observed solution no matter what original
initial condition was chosen for it.
In theory, one could then integrate the approximate solution
forward in time to obtain a good approximation of the current
physical state.  This approximation would then serve as an initial condition
for subsequent forecasts.

{{The auxiliary system described above was first proposed as an approach to
data assimilation for the model problem of the
two-dimensional incompressible Navier--Stokes equations by
Azouani, Olson, and Titi in~\cite{azouani:olson:titi}.
In that work, exponential convergence of the approximating
solution to the observed solution was shown under general
conditions in which the observations were assumed to be
taken continuously and instantaneously in time.
By now this approach has been studied for
several other physical systems such as the one-dimensional
Chaffee-Infante equation,
the two-dimensional Boussinesq,
the three-dimensional Brinkman-Forchheimer extended Darcy equations,
the three-dimensional B\'enard convection in porous media, and the three-dimensional Navier--Stokes $\alpha$-model
(cf. \cite{alb:lopes:titi, azouani:titi, azouani:olson:titi,
farhat:jolly:titi, farhat:lunasin:titi1, farhat:lunasin:titi2, mark:titi:trab}). Notably, Farhat, Lunasin, and Titi in \cite{farhat:lunasin:titi3}, recently verified, in the case of the three-dimensional planetary geostrophic model,
an earlier conjecture of Charney that posited that in simple atmospheric models, the temperature history determines all other
state variables.
The effects of noisy data were studied
by Bloemker, Law, Stuart, and Zygalakis \cite{blsz}
and Bessaih, Olson and Titi \cite{bessaih:olson:titi}. {\vincent{A case related to the study undertaken by this paper}}, where observations are taken at discrete moments in time, rather than continuously, {\vincent{and with systematic deterministic errors}}, was studied in \cite{foias:mondaini:titi}, while fully discretized versions were considered in \cite{ibdah:mondaini:titi}.  Postprocessing methods were also applied to further ameliorate errors in this downscaling algorithm and in particular, obtain error bounds which are uniform-in-time (cf. \cite{mondaini:titi}).
See also \cite{bergemann:reich} for a study into the
continuous-time extended Kalman-Bucy filter in the setting of stochastic nonlinear ODEs.
Observational
measurements that have been blurred in time are studied here.}}


 In continuation of the work in \cite{jmt}, we
combine a feedback control based on time-averaged modal
observables with the dynamics of the $2\pi$-periodic
subcritical SQG equation to obtain
	\begin{align}\label{fb:sqg:prelim}
		\begin{cases}
			\bdy_t\eta+\kap\Lam^\gam\eta+v\cdotp{\del}\eta
	=f-\mu
\big(
		{J}_h^{\de}(\eta)
		-{J}_h^{\de}(\tht)
\big),
			\\ v=\Ri^\perp\eta, \quad
                        \eta(x,t)\big{|}_{t\in(-2\de, 0]}=g(x,t).
		\end{cases}
	\end{align}
Here $\mu$ is a relaxation parameter, $J_h^\de(\theta)$
represents an idealized interpolant based on
modal measurements with {\vincent{observation resolution $h$}}
{\eric along with a moving time average}
over intervals of
width $\delta$ that
{\eric represents the blur}
intrinsic to the measuring device used to obtain the data.
It is natural to suppose that the observed solution, $\theta$,
represents the long-time evolution of the SQG equations,
which is to say that $\tht$ belongs to the global attractor
and therefore exists backward in time for all $t<0$.
For our analysis, however, it is sufficient to go back
only as far as $t=-2\delta$.  We therefore make the milder
assumption that
$\theta(\cdot,-2\delta)$ {\vincent{belongs to an absorbing ball for \req{sqg} with a sufficient regularity}}.
Note also that in order to construct the data assimilation
algorithm given by (\ref{fb:sqg:prelim}), we have assumed that
{\eric the SQG equation}
is known in addition to the
exact value of $\kappa$.
What is not known, of course, is the initial condition
for $\eta$ represented by the function $g(x,t)$.
Theoretically speaking one might as well
take $g(x,t)=0$; however, any $2\pi$-periodic function with
with mean zero that lies in {\vincent{the aforementioned absorbing set}} would be fine.
Therefore, there may be better choices for $g$ in practice.
In particular,
if we take $g(x,t)=\theta(x,t)$ for $t\in(-2\delta,0]$,
then $J_h^\de(\eta)=J_h^\de(\tht)$ in \req{fb:sqg:prelim},
so that $\eta(x,t)=\theta(x,t)$, for all $t>0$; {\vincent{we refer the reader to Section \ref{sect:pf1} to help clarify this fact}}.
Although there would be no need for data
assimilation if $\theta(x,t)$ were already known,
this cancellation is necessary {\eric to obtain the
important mathematical property that, in the absence of noise
or model error,} $\eta$
{\eric exactly} synchronizes with $\theta$ over time.


We will assume that {\eric equation \req{fb:sqg:prelim}
governs} the evolution of the
approximating solution, $\eta$, used in our analysis of data
assimilation for the SQG equation with observations that have
been blurred in time and with
$2\pi$-periodic boundary conditions over $\T^2$.
We will treat the subcritical case, when $\gam\in(1,2)$.
Our main results consist of the following two theorems:
\begin{enumerate}
	\item The data assimilation equations given by \req{fb:sqg:prelim} are well posed (Theorem \ref{thm1});
\smallskip
	\item For $h$ sufficiently small, there exists a choice of $\mu$ and $\delta$, for which the differences between $\eta$ and $\tht$ vanish over time (Theorem \ref{thm2}).
\end{enumerate}
Note that treating the critical case $\gamma=1$ would, of course, also be
very interesting for any type of {\eric observational data.
However, this is beyond the scope of our present analysis.}

We
{\eric defer formal}
statements of our theorems
{\eric to Section \ref{sect:state}, after}
we have defined
{\eric the mathematical setting of our problem}
in Section \ref{sect:prelim}.
Let us point out, however, that the presence of the moving time average
introduces certain analytical difficulties.
Firstly, it is difficult to control temporal oscillations in the
approximating solution that arise due to
deviations of the blurred-in-time observations
from the exact values of the reference solution.  For this, we must especially make use of more
delicate boundedness properties of the interpolant operator, which we identify and prove in Section \ref{sect:inter} and Appendix \ref{appC}, respectively.
Second, a suitable non-local Gronwall inequality is required
to control the difference between the approximating solution
the observed solution.
Theorem \ref{thm2} shows that these obstacles can indeed be surmounted
provided that $\delta$ is small enough.  In this regime, (\ref{fb:sqg:prelim}) achieves exact
asymptotic synchronization at an exponential rate and
therefore performs similarly to the case studied in \cite{jmt},
where the observations are not blurred in time.
Lastly, we emphasize that our approach to the analysis of this problem renders transparent
which errors arise from the delay and which arise from the blurring,
as well as the manner in which these errors transfer from one time-window to the next.
Because of this, we are able to capture {\eric mathematically}
the role of the size of the averaging window.


\section{Preliminaries}\label{sect:prelim}

\subsection{Function spaces: $L^p_{per}$, $V_\s$, $H_{{per}}^\s$, $\dot{H}_{{per}}^\s$, $C^\infty_{per}$}\label{hs:def}

Let $1\leq p\leq\infty$, $\s\in \R$ and $\T^2=\R^2/(2\pi\Z)=[-\pi,\pi]^2$.
Let $\mathcal{M}$ denote the set of real-valued Lebesgue measurable
functions over $\T^2$.  Since we will be working with periodic functions,
define
	\begin{align}\notag
		\mathcal{M}_{per}:=\{\phi\in\mathcal{M}:\phi(x,y)=\phi(x+2\pi,y)=\phi(x,y+2\pi)=\phi(x+2\pi,y+2\pi)\ \text{a.e.}\}.
	\end{align}
Let
{\eric $C^\infty(\R^2)$ be}
the class of functions which are infinitely
differentiable
{\eric on $\R^2$.}
Define $C^\infty_{per}(\T^2)$ by
	\begin{align}\notag
		C^\infty_{per}(\T^2):=
{\eric C^\infty(\R^2)}
\cap\mathcal{M}_{per}.
	\end{align}
For $1\leq p\leq\infty$, define the periodic Lebesgue spaces by
	\begin{align}\notag
		L^p_{per}(\T^2):=\{\phi\in\mathcal{M}_{per}: \Sob{\phi}{L^p}<\infty\},
	\end{align}
where
	\begin{align}\notag
		\Sob{\phi}{L^p}:=\left(\int_{\T^2} |\phi(x)|^p\ dx\right)^{1/p},\quad 1\leq p<\infty,\quad\text{and}\quad \Sob{\phi}{L^\infty}:=\ \stackrel[x\in\T^2]{}{\esssup}|\phi(x)|.
	\end{align}
Let us also define
	\begin{align}\label{Z:space}
		\mathcal{Z}:=\{\phi\in L^1_{per}:\int_{\T^2} \phi(x)\ dx=0\}.
	\end{align}

{\eric For $\phi\in L_{\it per}^1(\T^2)$ let}
$\hat{\phi}(\mathbf{k})$ denote the Fourier coefficient of $\phi$ at wave-number $\mathbf{k}\in\Z^2$, i.e.,
	\begin{align}\notag
		\hat{\phi}(\mathbf{k}):=\frac{1}{4\pi^2}\int_{\T^2}e^{-i\mathbf{k}\cdotp x}\phi(x)\ dx.
	\end{align}
For any real number $\s\geq0$, define the homogeneous Sobolev space, $\Hdper(\T^2)$, by
	\begin{align}\label{Hdper}
		\Hdper(\T^2):=\{\phi\in L^2_{per}(\T^2)\cap\mathcal{Z}:\quad \Sob{\phi}{\dot{H}^\s}<\infty\},
	\end{align}
where
	\begin{align}\label{hdper:norm}
				\Sob{\phi}{\dot{H}^\s}^2:=4\pi^2\sum_{\mathbf{k}\in\Z^2\smod\{\mathbf{0}\}}\abs{\mathbf{k}}^{2\s}|\hat{\phi}(\mathbf{k})|^2.
	\end{align}
{\eric Similarly, for}
$\s\geq0$, we define the inhomogeneous Sobolev space, $\Hper(\T^2)$, by
	\begin{align}\label{Hper}
		\Hper(\T^2):=\{\phi\in L^2_{per}(\T^2):\Sob{\phi}{H^\s}<\infty\},
	\end{align}
where
	\begin{align}\label{hper:norm}
		\Sob{\phi}{{H}^\s}^2:=4\pi^2\sum_{\mathbf{k}\in \Z^2}(1+\abs{\mathbf{k}}^2)^{\s}|\hat{\phi}(\mathbf{k})|^2.
	\end{align}

Let $\mathcal{V}_0\sub\mathcal{Z}$ denote the set of trigonometric
polynomials with mean zero over $\T^2$ and set
	\begin{align}\label{hper}
		V_\s:=\overline{\mathcal{V}_0}^{H^\s},
	\end{align}
where the closure is taken with respect to the norm given by \req{hper:norm}.
Observe that the mean-zero condition can be equivalently
stated as $\hat{\phi}(\mathbf{0})=0$.  Thus, $\Sob{\cdotp}{\dot{H}^\s}$ and $\Sob{\cdotp}{H^\s}$
are equivalent as norms over $V_\s$.  Moreover, by  Plancherel's theorem we have
	\begin{align}\notag
		\Sob{\phi}{\dot{H}^\s}=\Sob{\Lam^\s \phi}{L^2}.
	\end{align}
Finally, for $\s\geq0$, we identify $V_{-\s}$ as the dual space, $(V_\s)'$, of $V_\s$,
which can be characterized as the space of all bounded linear functionals, $\psi$, on $V_\s$
{\eric represented by the Fourier coefficients $\hat\psi({\bf k})$
with duality paring
$$
	\langle \psi,\phi\rangle =
		4\pi^2 \sum_{{\bf k}\in \Z^2\setminus{\bf 0}}
		{\overline{\hat\psi({\bf k})}}\cdot\hat\phi({\bf k})
\quad\hbox{such that}\quad
		\Sob{\psi}{\dot{H}^{-\s}}
		= 4\pi^2 \sum_{{\bf k}\in \Z^2\setminus{\bf 0}}
			|{\bf k}|^{-2\sigma}\big|\hat\phi({\bf k})\big|
		<\infty.
$$
}%
{\eric Given our use of non-dimensional variables and
the $2\pi$ spatial periodicity of our functions,
the Poincar\'e inequality may be written
with a non-dimensional constant equal to one as
\begin{equation}\label{pineq}
	\|\phi\|_{\dot H^{\sigma'}} \le
	\|\phi\|_{\dot H^{\sigma}}
\qquad\hbox{for}\qquad \sigma'\le\sigma.
\end{equation}
Moreover,} we have the following continuous embeddings
	\[
		V_\s\imb V_{\s'}\imb V_0\imb V_{-\s'}\imb V_{-\s}
{\eric \qquad\hbox{when}\qquad} 0\leq\s'\leq\s.
	\]

\begin{rmk}
Since we will be working over $V_\s$ and $\Sob{\cdotp}{\dot{H}^\s}$, $\Sob{\cdotp}{H^\s}$
determine equivalent norms over $V_\s$,
we will often denote $\Sob{\cdotp}{\dot{H}^\s}$
simply by $\Sob{\cdotp}{H^\s}$ for convenience.  Similarly, we will often abuse
notation and denote ${L} ^p_{per}({\T^2})$ simply by $L^p$.
\end{rmk}

\subsection{General Interpolant Observables}\label{sect:inter}
We will consider general interpolant observables,
$J_h$, which are defined as those which satisfy certain boundedness
and approximation-of-identity properties.
The canonical examples of such observables include projection
onto local spatial averages or projection onto
finitely many Fourier modes.
It was shown in \cite{jmt}
that such projections do in fact satisfy the properties we impose
{on} $J_h$.

Let $0<h<\pi/3$ and $1\leq q\leq p\leq\infty$.  Let $J_h:{L}^p(\T^2)\goesto {L}^p(\T^2)$ be a linear operator satisfying
	\begin{align}
		\sup_{h>0}\Sob{J_h\phi}{L^p}&\leq C\Sob{\phi}{L^p}, \quad\quad\quad\quad\quad\quad\quad\hspace{9pt}\label{Ih:base0}\\
		\Sob{J_h\phi}{L^p}&\leq Ch^{2(1/p-1/q)}\Sob{\phi}{L^q}\label{Ih:base1},
	\end{align}
where $C>0$ represents a constant independent of $\phi, h$.
{\eric Note that $1/p-1/q< 0$ when $q<p$
in which case the bound in (\ref{Ih:base1}) gets worse as $h$ becomes smaller.}
In addition to
{\eric \req{Ih:base0} and \req{Ih:base1},}
we will also suppose that $J_h$ satisfies the
following approximation-of-identity properties
	\begin{align}
		\Sob{\phi-J_h\phi}{L^2}\leq Ch^{\be}\Sob{\phi}{\dot{H}^\be},\quad \text{and}\quad
		\Sob{\phi-J_h\phi}{\dot{H}^{-\be}}\leq Ch^{\be}\Sob{\phi}{L^2},\quad \be\in(0,1].\label{TI}
	\end{align}
	
We will also require $J_h$ to satisfy {\eric some}
boundedness properties.
We verify in Appendix~\ref{appC} that
{\eric these properties}
hold for local spatial
averages.  They {\eric also hold} for spectral projection, {\eric that is}, projection
onto finitely many {\eric lowest} Fourier modes (see Remark \ref{rmk:LPproj}). To state
{\eric these boundedness properties}, we will adopt the following notation.
{\eric For $\beta_1$ and $\beta_2$ non-negative integers
we let $D^\be:=\bdy_1^{\be_1}\bdy_2^{\be_2}$ where $\be_1+\be_2=\be$,
while if $\beta_j\ge 0$ are real
then $D^\be:=\bdy_1^{\lfloor\be_1\rfloor}
	\bdy_2^{\lfloor\be_2\rfloor}\Lam^{\be
		-\lfloor\be_1\rfloor-\lfloor\be_2\rfloor}$.
Here $\lfloor\beta\rfloor$
represents the greatest integer less or equal $\beta$.}
Finally, if $\be\in(-2,0)$, then
$D^\be:=\Lam^{\be}$, i.e., the Riesz potential.

Now, given $\al\geq1$, let $\eps(\al)$ be as in Proposition \ref{prop:pou} $(v)$ when $\al\in[1,2)$ and identically $0$ otherwise.
{\eric Let $C_\alpha>0$}
be a sufficiently large constant, depending possibly on $\al$, and define
	\begin{align}\label{typeI:const11}
		C_I(\al,h):=\begin{cases}
{\eric C_\alpha}
		\left(\frac{2\pi}{h}\right),& \al<1,\\
{\eric C_\alpha}
		\left(\frac{2\pi}{h}\right)^{2+|\al|-\eps(\al)},&\al\geq1{\eric .}
				\end{cases}
	\end{align}
We assume that
	\begin{align}
	\Sob{J_h\phi}{\dot{H}^\rho(\T^2)}&\leq C_I(\be,h)
				h^{-(\rho-\be)}\Sob{\phi}{\dot{H}^{\be}(\T^2)},
	\quad (\rho,\be)\in[0,\infty)\times[0,2)
,\label{TI:3}\\
		\Sob{J_h\phi}{\dot{H}^\rho (\T^2)}&\leq Ch^{-\rho}(h^{\be}\Sob{\phi}{\dot{H}^\be}+\Sob{\phi}{L^2(\T^2)}),
	\quad(\rho,\be)\in[0,\infty)\times(-2,0]\label{TI:4}
,\\
		\Sob{J_h\phi}{\dot{H}^\rho(\T^2)}&\leq C_I(|\rho|,h)h^{-(\rho-\be)}\Sob{\phi}{\dot{H}^{\be}(\T^2)},\quad (\rho,\be)\in(-2,0)\times(-\infty,0],\label{TI:5}\\
		\Sob{J_hD^\be\phi}{\dot{H}^\rho(\T^2)}&\leq C_I(|\rho|,h)h^{-(\rho+\be-\be')}\Sob{\phi}{\dot{H}^{\be'}(\T^2)},\notag\\
	&\qquad\qquad\qquad\qquad
	(\rho,\be,\be')\in(-2,0)\times(-2,\infty)\times(-\infty,\be],\label{TI:8}\\
		\Sob{J_hD^\ell\phi}{\dot{H}^\rho(\T^2)}&\leq C_I(|\rho|,h)h^{-1-\rho-\ell}\Sob{\phi}{L^1(\T^2)},\quad (\rho,\ell)\in(-2,0)\times\Z_{\geq0},\label{TI:9}.
	\end{align}
	
{\eric We again emphasize that the above} properties
are consistent with those satisfied by
the projection onto local spatial averages (see \req{vol:elts} and
\req{vol:elts:shift} in Appendix \ref{appC}).
{\eric Furthermore, we again} point
out that they
are also consistent with those satisfied by the spectral projection,
up to possibly different
{\eric constants}
(See Remarks \ref{rmk:spec:proj}
and \ref{rmk:LPproj}).  For clarity of exposition, our analysis will
be performed with the
{\eric constants detailed}
above, though the conclusions are also
true for $J_h$ given by spectral projection.

\begin{rmk}\label{rmk:morebdd}
We are able to prove other boundedness properties in Appendix \ref{appC} in addition to the ones shown above.  While our analysis requires us only to invoke properties \req{Ih:base0}-\req{TI:9}, the additional boundedness properties asserted in Proposition \ref{prop:Jh:bddness} may find use in other applications.
\end{rmk}

\begin{rmk}\label{rmk:spec:proj}
In the case where $J_h$ is given by the Littlewood-Paley spectral projection, i.e., projection onto Fourier modes $\ls2^{1/h}$, then we replace $C_I(\al,h)$ everywhere above by $C_{S}(\al,h)$ according to the rule
	\begin{align}\notag
		C_{I}(\al,h)h^{r}\mapsto C_{S}(\al,h)h^r:=\begin{cases} C,& r\geq0,\\
						Ch^r,&r<0.
					\end{cases}\quad\text{and}\quad \til{C}_{S}:=C.
	\end{align}
Note that $\al=\al(p)$ implicitly.  One may thus refer to operators $J_h$
with {\eric constants} $C_I$ as ``{\eric Type} I operators" and those with prefactors
$C_{S}$ as ``{\eric Spectral Type I} operators."
Observe that in general we have $C_{S}\ls C_I$,
so all {\eric Spectral Type I} operators are automatically {\eric Type} I
operators.
{\eric We further observe that the Type II operators defined
in~\cite{azouani:olson:titi},
see also~\cite{bessaih:olson:titi},
using nodal-point measurements of the velocity field
in physical space do not satisfy the above bounds.}
\end{rmk}

\begin{rmk}
Note that in the estimates we perform below, the constant $C>0$
appearing in \req{typeI:const11} may change line-to-line when invoking
the above properties.  Nevertheless, it can be fixed to be sufficiently
large in the statement of the theorems where such constants appear.
\end{rmk}

\subsection{Time-averaged Interpolant Observables}


Suppose $\phi=\phi(x,t)$.  We define the time-averaged general
interpolant operator, $J_h^\de$, by
	\begin{align}\label{lp:interpolant}
		(J_h^\de\phi)(x{\eric ,t})
    :=\frac{1}{\de}\int_{t-2\de}^{t-\de}(J_h\phi)(x,s)\ ds
	\end{align}
%
Due to the time-averaging, one must also control errors that
arise from temporal deviations of the time-average from the instantaneous value
value.  Indeed, observe that by the mean value theorem
and by commuting $\bdy_\tau$ with $J_h$ we have
	\begin{align}\label{mvt}
		\phi-J_h^\de\phi
			&=(\phi-J_h\phi)+\frac{1}{\de}\int_{t-2\de}^{t-\de}\int_s^t
			 J_h\bdy_\tau\phi(x,\tau)\ d\tau\ ds.
	\end{align}
We will make crucial use of \req{mvt} when
we perform the a priori estimates.

\begin{rmk}\label{rmk:delay}
{\eric It may seem more natural to represent blurred-in-time
measurements at time $t$ by an average of the form
$$
	(I_h^\de\phi)(x,t):={1\over\delta}
		\int_{t-\delta/2}^{t+\delta/2} ({\eric J_h}\phi)(x,s)\,ds.
$$
However, in this case the corresponding a feedback term
obtained by using $I_h^\de(\eta)$ in place of $J_h^\de(\eta)$
in~\req{fb:sqg:prelim} would violate causality by introducing
an integral over times in the future.
We emphasize that
the same interpolant operator
must be used in the feedback as used for the measurements
in order to maintain the property that
$g=\tht$ for $t\in(-\de,0]$ implies $\eta=\tht$ for all times $t>0$
in the future.
Therefore, the best we could do is insert the
measurement $I_h^\delta(\phi)$ into the model delayed
in time by $\delta/2$.
This approach was taken in \cite{blocher}
and \cite{blocher18} for the Lorenz equations.
In the present work,
an additional delay has been inserted
into the definition of $J_h^\delta \phi$
to make the analysis more convenient.
This allows the feedback control to be treated
as a time-dependent force, thereby transforming
what would have been partial integro-differential equations
into merely partial differential equations.
While any additional delay would achieve the same effect, for
simplicity we choose its order
to be $\delta/2$ which is the same as the delay already dictated by causality.
}\end{rmk}

\subsection{Calculus inequalities}

We will make use of the following bound for the fractional Laplacian,
which can be found for instance in \cite{const:gh:vic, ctv,
  ju:qgattract}.

\begin{prop}\label{lb}
Let $p\geq2$, $0\leq\gam\leq2$, and $\phi\in C^\infty(\T^2)$. Then
	\begin{align}\notag
{\eric
		\frac{2}p\Sob{\Lam^{\gam/2}
                  (\phi^{p/2})}{L^2}^2
		\le
		\int_{\T^2}|\phi|^{p-2}(x)\phi(x)\Lam^\gam
                \phi(x)\ dx.}
	\end{align}
\end{prop}

We will also make use of the following calculus inequality for
fractional derivatives (cf.  \cite{kato:ponce, kato:ponce:vega} and
references therein):

\begin{prop}\label{prod:rule}
Let $\phi, \psi\in C^\infty(\T^2)$, $\be>0$, and $p\in(1,\infty)$.
Then we have that
	\begin{align}\notag
		\Sob{\Lam^\be(\phi\psi)}{L^p}\leq
                C\Sob{\psi}{L^{p_1}}\Sob{\Lam^\be \phi}{L^{p_2}}+C
                \Sob{\Lam^\be \psi}{L^{p_3}}\Sob{\phi}{L^{p_4}},
	\end{align}
where $1/p=1/p_1+1/p_2=1/p_3+1/p_4$, and $p_2,p_3\in(1,\infty)$, for a
sufficiently large constant $C>0$ that depends only on $\s,p,p_i$.
\end{prop}

Finally, we will frequently apply the following interpolation
inequality, which is a special case of the Gagliardo-Nirenberg
interpolation inequality and can be proven with Plancherel's theorem
and the Cauchy-Schwarz inequality:

\begin{prop}\label{interpol}
Let $\phi\in \dot{H}^{\be}_{per}(\T^2)$ and $0\leq \al\leq \be$.  Then
	\begin{align}\label{gn:ineq}
		\Sob{\Lam^{\al} \phi}{L^2}\leq C\Sob{\Lam^{\be}
                  \phi}{L^2}^{\frac{\al}{\be}}\Sob{\phi}{L^2}^{1-
                  \frac{\al}{\be}},
	\end{align}
where $C$ depends on $\al,\be$.
\end{prop}

\subsection{Well-posedness and Global Attractor of the SQG
	equation}\label{sect:ball}
Let us recall the following well-posedness results of the SQG
equation.  In \cite{const:wu:qgwksol} it was shown that global strong
solutions exist and that weak solutions are unique in the class of
strong solutions.

\begin{prop}[Global existence]\label{prop:glob}
Let $1<\gam\leq 2$, and $\s>2-\gam$.  Given $T>0$, suppose that
$\tht_0\in V_\s$ and $f$ satisfies
	\[
		f\in L^2(0,T; V_{\s-\gam/2})\cap
                L^1(0,T;{L}_{per}^p(\T^2)),
	\]
where $1-\s\leq2/p<\gam-1$.  Then there is a weak solution $\tht$ of
\req{sqg} such that
	\begin{align}\notag
		\tht\in L^\infty(0,T;V_\s)\cap
                L^2(0,T;V_{\s+\gam/2}).
	\end{align}
\end{prop}

\begin{prop}[Uniqueness]\label{prop:uniq}
Let $T>0$ and $1<\gam\leq2$.  Suppose that $\tht_0\in
{L}_{per}^2(\T^2)\cap\mathcal{Z}$ and $f\in L^2(0,T;V_{-\gam/2})$.  Then for $p\geq1, q>0$ satisfying
	\[
		\frac{1}p+\frac{\gam}{2q}=\frac{\gam-1}2,
	\]
there is at most one solution to \req{sqg} such that
	$
		\tht\in L^q(0,T;{L}_{per}^p(\T^2)).
	$
\end{prop}

Let us recall the following estimates for the reference solution $\tht$ (cf. \cite{ctv, ju:qgattract, resnick}).


\begin{prop}\label{prop:sqg:ball}
Let $\gam\in(0,2]$, $\s>2-\gam$, and $\tht_0\in V_\s,
  f\in V_{\s-\gam/2}\cap{L}_{per}^p(\T^2)$.
  Then there exists a constant $C>0$ such that for any $p
  \geq2$ satisfying $1-\s<2/p<\gam-1$, we have
	\begin{align}\label{fp}
		\Sob{\tht(t)}{L^p}\leq
                \left(\Sob{\tht_0}{L^p}-\frac{1}{C}{\FLp}\right)
					e^{-C{{{\kap}}}t}+\frac{1}{C}
                     {\FLp},\quad
                     {\FLp}:=\frac{1}{{\kap}}\Sob{f}{L^p}.
	\end{align}
Moreover, if $\tht_0\in {L}^2_{per}(\T^2)$ and $f\in V_{-\gam/2}$, then any weak solution $\tht$ of
\req{sqg} satisfies
	\begin{align}\label{fgam}
		\Sob{\tht(t)}{L^2}^2\leq
                \left(\Sob{\tht_0}{L^2}^2-{\FHg^2}\right)e^{-{{{\kap}}}
                  t}+{\FHg^2}, \quad
                \FHg:=\frac{1}{{\kap}}\Sob{f}{H^{-\gam/2}}.
	\end{align}
\end{prop}

{
It was shown in \cite{ju:qgattract} for the subcritical
range $1<\gam\leq2$, that equation \req{sqg} has an absorbing ball in
$V_\s$
and corresponding global attractor
$\A\sub V_\s$ when $\s>2-\gam$.
In other words, there is
a bounded set $\mathcal{B}\sub V_\s$
characterized by the property that for any $\tht_0\in V_\s$,
there exists $t_0>0$ depending on $\Sob{\tht_0}{H^\sigma}$ such that
$S(t)\tht_0\in\mathcal{B}$ for all $t\geq t_0$.
Here $\{S(t)\}_{t\geq0}$
denotes the semigroup of the corresponding dissipative equation.
}

\begin{prop}[Global attractor]\label{prop:ga}
Suppose that $1<\gam\leq2$ and $\s>2-\gam$.  Let $f\in
V_{\s-\gam/2}\cap {L} _{per}^p(\T^2)$, where
$1-\s<2/p<\gam-1$.  Then \req{sqg} has an absorbing ball
$\mathcal{B}_{H^\s} $ given by
	\begin{align}\label{sqg:hs:ball}
		\mathcal{B}_{H^\s}:=\{\tht_0\in\dot{H}^\s_{per}:
                \Sob{\tht_0}{H^\s}\leq \THs\},
	\end{align}
for some $\THs<\infty$.  Moreover, the solution operator $S=S_f$
of \req{sqg}
{\eric given by}
$S(t)\tht_0=\tht(t)$ {\eric for $t\ge0$}
defines a semigroup in the
space $V_\s$ and possesses a global attractor $\A\sub V_\s$, i.e.,
$\A$ is a compact, connected subset of
$V_\s$ satisfying the following properties
	\begin{enumerate}
		\item $\A$ is the maximal bounded invariant set;
		\item $\A$ attracts all bounded subsets in
                  $V_\s$ in the topology of $
                  \dot{H}^\s_{per}$.
	\end{enumerate}
\end{prop}

\section{Standing Hypotheses and Statements of main theorems}\label{sect:state}
We will work under the following assumptions for the remainder of the
paper.

\begin{stand}\label{std:hyp}
Assume the following:
\begin{enumerate}[(H1)]
	\item $1<\gam<2$;
	\item $\s\in(2-\gam,\gam]$;
	\item $p\in[1,\infty]$ such that $1-\s<2/p<\gam-1$;
	\item$ f\in V_{\s-\gam/2}\cap L^p$, time-independent;
	\item $\tht_{-2\de}\in\mathcal{B}_{H^\s}$;
	\item $g\in C((-2{\de}, 0];V_{\max\{\s,\gam/2\}})\cap
          L^2((-2{\de}, 0]; V_{\s+\gam/2})$;
	\item $0<h<\pi/4$.
\end{enumerate}
\end{stand}

Observe that $(H1)$ expresses the subcritical range of dissipation,
 while $(H2)-(H5)$ ensure that we are in a regime of global strong
  solutions for \req{sqg} and that the global attractor exists.

Also observe that since $\gam<2$, the range for $\s$ in $(H2)$
 covers the natural spatial regularity class for
 strong solutions, e.g. $H^\gam$.

On the other hand, from $(H1)-(H5)$, Propositions \ref{prop:sqg:ball} and
\ref{prop:ga} imply that
	\begin{align}\label{tht:bounds}
		\TLt:=\sup_{t>-2\de}\Sob{\tht(t)}{L^2}<\infty\quad
                \text{and}\quad\TLp:=\sup_{t>-2\de}
                \Sob{\tht(t)}{L^p}<\infty.
	\end{align}
In particular, it immediately follows from \req{Ih:base0} that
	\begin{align}\label{tht:bc}
		\sup_{t>-2\de}\Sob{J_h^{{\de}}\tht(t)}{L^q}^2\leq
                {C_J}\Tht_{L^q}^2,\quad q\in[1,\infty],
	\end{align}
and from \req{TI:3} that
	\begin{align}\label{tht:bc2}
		\sup_{t>-2\de}\Sob{J_h^{\de}\tht(t)}{H^\s}\leq {C_J}\Tht_{H^\s},
	\end{align}
for some constant ${C_J}>0$.
Also, for $1\leq q\leq\infty$ and $\al\in\R$, let us define
	\begin{align}\label{gam}
		\Gam_{L^q}:=\sup_{t\in(-2\de,0]}\Sob{g(t)}{L^q}
		\quad\text{and}\quad\Gam_{H^\al}:=\sup_{t
                    \in(-2\de,0]}\Sob{g(t)}{H^\al}.
	\end{align}
Then for $p$ given by $(H3)$, the Sobolev embedding theorem
and $(H6)$ imply
	\begin{align}\label{gam:bounds}
		\Gam_{L^p}<\infty,\quad
                \Gam_{H^\s}<\infty
		{\eric \quad\text{and}\quad}
\Gam_{H^{\gam/2}}<\infty.
	\end{align}

Finally, we give
{\eric exact mathematical} statements of our {\eric main} results.
	
\begin{thm}\label{thm1}
Let $\tht$ be the unique
    global strong solution of \req{sqg} corresponding to initial data
    $\tht_{-2\de}$ having zero mean over $\T^2$.  Then under the
    Standing Hypotheses, 
for all $T>0$, there exists a unique strong solution
$\eta\in L^\infty(0,T;\dot{H}_{per}^\s(\T^2))\cap L^2(0,T;\dot{H}_{per}^{\s+ \gam/2}(\T^2))$ satisfying
\req{fb:sqg:prelim} with $\eta(\cdotp,0)=g(\cdotp,0)$.
\end{thm}

\begin{thm}\label{thm2}
Under the hypotheses of Theorem \ref{thm1}, there exists
constants $c_0, c_0'>0$ such that
if $h, \mu$ satisfy
	\begin{align}\label{mu:sync:cond}
	\frac{1}{c_0'}\left(\frac{\TLp}
                {\kap}\right)^{\gam/(\gam-1-2/p)}\leq\frac{\mu}{\kap}\leq\frac{1}{c_0}h^{-\gam},
	\end{align}
and $\de>0$ is chosen sufficiently small, depending on $h$, then the solution
$\eta$ given by (\ref{fb:sqg:prelim}) satisfies
	\begin{align}\label{synch:thm2}
		\Sob{\eta(t)-\tht(t)}{L^2}^2\leq{O}
			(e^{-\lam_0\mu(t-2\de)}),\quad t>2\de,
	\end{align}
for some constant $\lam_0\in(0,1)$.
\end{thm}

\begin{rmk}\label{thm:de:cond}
Note that the condition that $\de>0$ be sufficiently small can be described precisely by simultaneously satisfying \req{de:small:l2} and \req{de:sync:avg} below.
\end{rmk}

\begin{rmk}\label{thm:modescase}
As we pointed out in Remark \ref{rmk:spec:proj}, since {\eric Spectral Type I}
operators satisfy all the properties of {\eric Type} I operators,
both Theorem \ref{thm1} and
\ref{thm2} are also valid for {\eric Spectral Type I} operators.
In particular, they are
valid when $J_h$ is given by projection onto finitely many Fourier modes.
\end{rmk}

\begin{rmk}The relationship between
the full three-dimensional quasi-geostrophic equations
and the SQG equation
implies that being able to approximate $\theta$
by $\eta$, as in the conclusion of Theorem \ref{thm2}, is the same
as synchronizing the corresponding three-dimensional solutions
in which the potential vorticity is identically zero and the
vertical motion eliminated.
Therefore, in a way analogous to the discussion in \cite{jmt},
our theorem provides an example where time-averaged data
collected on a two-dimensional surface is sufficient to obtain
synchronization in a three-dimensional domain.
\end{rmk}

Before we move on to the a priori analysis, we will set forth the
following convention for constants.
\begin{rmk}
In the estimates that follow below, $c$ {and} $C$ will generically denote
positive constants, which depend only on other
non-dimensional scalar quantities, and may change line-to-line in the
estimates.  We emphasize that in the estimates we perform below, the
constants $c$ {and} $C$ may change in magnitude from
line-to-line, but as the equations were fully non-dimensionalized
from the beginning they will never carry any physical dimensions.
\end{rmk}

\section{A priori estimates}\label{sect:a priori2}

\subsection{Initial value problem and Proof of Theorem \ref{thm1}} \label{sect:pf1}
We recouch \req{fb:sqg:prelim} as a sequence of initial value problems over
consecutive time intervals.  Once we have defined the setting
properly, we may immediately prove Theorem \ref{thm1}
by appealing to Propositions \ref{prop:glob} and \ref{prop:uniq}.

Observe that owing to the delay in the interpolant operator,
$J_h^{{{\de}}}$, we must initialize the averaging process.  By
$(H1)-(H5)$ and Proposition \ref{prop:ga}, we may assume that $\tht$
is the strong solution of \req{sqg} with initial data starting at
$t=-2{{\de}}$ such that $\tht_{-2\de}\in \mathcal{B}_{H^\s}$.

For any $k\geq-2$ set
	\begin{align}\label{notation}
		I_{-2}:=\emptyset,\quad
                I_{-1}:=(-2{{\de}},0],\quad\text{and}\quad
                  {{\de}_k}:=k{{\de}},\quad
                  I_{k}:=({{\de}_k},{{\de}_{k+1}}],\quad
                    \text{for}\ k\geq 0.
	\end{align}
Let $\eta^{(-1)}(\cdotp, t)=g(\cdotp, t)$ for $t\in I_{-1}$.  Then we
may express a solution, $\eta$, of
	\begin{align}\label{fb:sqg:avg}
			\bdy_t\eta+{{{\kap}}}\Lam^\gam\eta+v\cdotp{\del}\eta=f-{{{\mu}}}
                        J_h^{{{\de}}}(\eta- \tht),\quad
                        v={{}}\Ri^\perp\eta, \quad
                        \eta(x,t)\big{|}_{t\in I_{-1}}=g(x,t).
	\end{align}
as the sum
	\begin{align}\notag
		\eta(x,t):=\sum_{k\geq-1}\ek(x,t)\chi_{I_{k}}(t),
	\end{align}
where for each $k\geq0$, $\ek$ satisfies:
	\begin{align}\label{fb:k}
		\begin{split}
			&\bdy_t\ek+{{{\kap}}}\Lam^\gam\ek+\vk\cdotp{\del}\ek=f-{{{\mu}}}
                  J_h^{{{\de}}}(\ek- \tht),\quad t\in
                  I_{k},\\ &\vk=\Ri^\perp\ek, \quad
                  \ek(x,t)\big{|}_{t\in I_{k-1}\cup
                    I_{k-2}}=\eta^{(k-1)}(x,t).
		\end{split}
	\end{align}
Hence, over each interval $I_k$ we may view the term, $J_h^{{\de}}
\eta^{(k)}$, in \req{fb:k} as a smooth, time-dependent forcing term
and \req{fb:k} as an initial value problem over $I_k$ with initial
data $\eta_0(x)=\eta(x,{\de}_k)$.  The proof of Theorem \ref{thm1}
follows readily.

\begin{proof}[Proof of Theorem \ref{thm1}]
We proceed by induction on $k$.  For $k=0$, from $(H6)$ we have that
$\eta(\cdotp, 0)= g(\cdotp, 0)\in V_\s$.  Since we assume
the Standing Hypotheses, we have that $J_h^\delta g=J_h^\de \eta^{(0)},
J_h^\de\tht\in L^2(0,T; V_{\s-\gam/2})\cap
L^1(0,T;{L}_{per}^p(\T^2))$ holds for all $T>0$ (by \req{Ih:base0} and \req{TI:3}), so that we may apply Proposition \ref{prop:glob} and
\ref{prop:uniq} to deduce existence and uniqueness of a strong
solution, $\eta^{(0)}$, over $I_0$ to \req{fb:k}.  Suppose unique
strong solutions to \req{fb:k} exist for all $\ell=0,\dots, k$.
Consider \req{fb:k} over $I_{k+1}$.
Observe that by
hypothesis $\eta^{(k+1)}(\cdotp, \de_{k+1})=\eta^{(k)} (\cdotp,
\de_{k+1})\in V_\s$ and $J_h^\de\eta^{(k+1)},
J_h^\de\tht\in L^2(\de_{k-1},\de_{k-1}+T;
V_{\s-\gam/2})\cap
L^1(\de_{k-1},\de_{k-1}+T;{L}_{per}^p(\T^2))$ {\eric hold}
once again by \req{Ih:base0} and \req{TI:3}.  Therefore, we apply Proposition
\ref{prop:glob} and \ref{prop:uniq} to guarantee existence and
uniqueness of a strong solution $\eta^{(k+1)}$ to \req{fb:k} over
$I_{k+1}$, completing the proof.
\end{proof}

In the
{\eric remainder of section \ref{sect:a priori2} we establish}
uniform-in-time
estimates for $\eta$ in $L^2$, $L^p$, and $H^\s$.  As we will see,
the synchronization property will rely crucially on these uniform estimates.
{\eric To} obtain uniform
$H^\s$ estimates, we perform a bootstrap from $L^2$ to $L^p$, then
from $L^p$ to $H^\s$.  Once we have collected the requisite uniform
bounds, we proceed {\eric to section \ref{sect:pf2} and} the proof of Theorem \ref{thm2}.

\subsection{Uniform $L^2$ estimates}\label{sect:l2}
{{} In this section, we will ultimately
{\eric obtain $L^2$}
estimates for the solution $\eta$ of \req{fb:sqg:avg} that
are uniform in time.
In this work, any bound of this type shall be
referred to as a {\eric ``good"} bound.}
The main result in this section {\eric is the ``good" bound stated
as Proposition \ref{prop:l2:avg} below.}
We emphasize
that the structure of the analysis in sections \ref{sect:temp:osc}, \ref{sect:l2:trans}, and \ref{sect:l2:avg} will be
{\eric mimicked in section}
\ref{sect:pf2} when we establish the synchronization property.

{\eric We begin by introducing some notation that will be convenient
when expressing the necessary bounds in our proofs.  Let}
	\begin{align}\label{M2:def:final}
		\til{R}_{L^2}^2:=\frac{\kap^2}{\mu^2}F_{H^{-\gam/2}}^2
		+{C_J}\Tht_{L^2}^2,\quad
                R^2_{L^2}:=\frac{\kap}\mu
                F_{H^{-\gam/2}}^2+{C_J}\Tht_{L^2}^2,\quad M_{L^2}^2:=\Gam_{2,1}^2+8R_{L^2}^2
	\end{align}
where $\Gamma_{2,k}$ is the function of $\delta>0$ given by
	\begin{align}\label{Gamk}
\Gam_{2,-1}:=\Gam_{L^2}\quad\hbox{and}\quad
		\Gam_{2,k}^2:=\Gam_{2,k-1}^2
			+C\frac{{\de}{\mu}^2}{{\kap}}\left(\Gam_{2,k-1}
                ^2+\til{R}_{L^2}^2\right)
	\quad\hbox{for}\quad
	 k\geq0.
	\end{align}
{\eric
Note that
$\Gamma_{2,k}$ and consequently $M_{L^2}^2$ are
increasing functions of $\delta$.
Therefore, any upper bounds given by the constants defined in
(\ref{M2:def:final}) and (\ref{Gamk}) for a particular
$\delta=\delta_0$
continue to hold when $\delta<\delta_0$.
We shall immediately make use of this property to show
that the hypotheses on $\delta$ in
Proposition \ref{prop:l2:avg} stated below
are not vacuous.}

\begin{prop}\label{prop:l2:avg}

There exist constants $c_0, c_1>0$, with $c_1$ depending on $c_0$,
such that if $h, {\mu}$ satisfy
	\begin{align}\label{m:mu:avg:l2}
		\frac{\mu}{\kap}\leq \frac{1}{c_0}h^{-\gam},
	\end{align}
and ${\de}$ is chosen {\eric such that}
	\begin{align}\label{de:small:l2}
		\begin{split}
		\delta \le \frac{1}{c_1}\frac{h^\gam}{\kap}\min\Bigg{\{}&1,
				{h^{\gam}}\frac{8R_{L^2}^2}{(M_{L^2}^2+
            \til{R}_{L^2}^2)},
    		\left(\frac{h}{2\pi}\right)\frac{1}{(1+\kap^{-1}h^{\gam-2})}\frac{R_{L^2}}{\left(1+M_{L^2}^2+R_{L^2}^2\right)}
			\Bigg{\}}
	\end{split}
		\end{align}
as well as
	\begin{align}\label{gron:cond:l2}
		\delta\le
		\frac{1}{c_1}\left(\frac{h}{2\pi}\right)\min
		\left\{\left(\frac{\mu h^\gam}{\kap}\right)^{1/2}
		\frac{h^2}{M_{L^2}},\frac{h^\gam}{\kap}\right\}
	\end{align}
where $\widetilde R_{L^2}^2$, $R_{L^2}^2$
and $M_{L^2}^2$ are given in (\ref{M2:def:final}),
then
	\begin{align}\label{m2:full}
		\Sob{\eta(t)}{L^2}^2\leq
    \left(\Gam_{2,1}^2-8{R_{L^2}^2}\right)e^{-({{{\mu}}}/2)(t-2\de)}
                +8{R_{L^2}^2}\quad\hbox{for}\quad t\geq2\de
	\end{align}
and
	\begin{align}\label{mgam2:avg}
		\frac{{{\kap}}}4\int_{I_k}\Sob{\eta(s)}{H^{\gam/2}}^2\ ds\leq
                \Gam_{2,1}^2+8R_{L^2}^2
			\le M_{L^2}^2
	\quad\hbox{for}\quad k\geq2.
	\end{align}
\end{prop}

{\eric Observe that both sides of the inequalities given
by (\ref{de:small:l2}) and (\ref{gron:cond:l2}) depend on $\delta$.
This is, as already mentioned, because $M_{L^2}^2$ depends on $\delta$.
However, since $M_{L^2}^2$ appears in the denominator of the right-hand
side and is an increasing function of $\delta$, it is easy to see
that there must be a $\delta>0$ which satisfies both these
inequalities.}

To prove {\eric Proposition \ref{prop:l2:avg}, we
{\eric employ}
three preliminary lemmas.  First,
in section \ref{sect:l2:rough}
{\eric we establish}
bounds in $L^2$ which are uniform in each time interval $I_k$,
but ultimately depend on $k$.
Throughout this work we will refer to any bounds that
depend on $k$ as {\eric ``{rough}"} bounds.
Such bounds are insufficient on their own but needed in
order to close estimates later. }
Then in section \ref{sect:temp:osc},
{\eric we establish}
time-derivative estimates to control the temporal
oscillations that emanate from the feedback term (see section
\ref{sect:temp:osc}).
The third
{\eric lemma is}
is a non-local
Gronwall inequality that ensures uniform bounds provided that the
window of time-averaging is sufficiently small; its proof is deferred
to Appendix \ref{appA}.  This Gronwall inequality will be
used again to establish the synchronization
property in 
section \ref{sect:pf2}.  We finally prove Proposition \ref{prop:l2:avg} in section
\ref{sect:l2:avg}.

\begin{rmk}\label{rmk:mu:exchange}
{\eric We will} often exchange the quantity $\mu$ for the quantity $\kap
h^{-\gam}$ via the relation~\req{m:mu:avg:l2}, in order to emphasize that
$\de$ and $\mu$ ultimately depend only on $h$ (and $\Tht_{L^p}$) alone.
\end{rmk}

\subsubsection{Rough $L^2$ estimates}\label{sect:l2:rough}
We will first establish the following ``rough" a priori bound.  We
omit most of the details, though they can easily be gleaned from the
proof of Proposition \ref{prop:l2:avg}.  An alternative form of
Lemma \ref{prop:bad} is {\eric given by} Corollary \ref{coro:bad}
stated below,
which will be convenient to use
{\eric in} the proof of
Proposition~\ref{prop:l2:avg} later.

\begin{lem}\label{prop:bad}
Let $F_{H^{-\gam/2}}, \Tht_{L^2}, \til{R}_{L^2}$ be given by
\req{fgam}, \req{tht:bounds}, \req{M2:def:final}, respectively.  There
exists a constant $C_0>0$, independent of $k$, such that
	\begin{align}\label{bad:Mest}
		\Sob{\eta(t)}{L^2}^2+{\kap}
	\int_{{{\de}_k}}^t\Sob{\eta(s)}{H^{\gam/2}}^2\ ds\leq
        \til{M}_{L^2} ^2(k,t),\quad t\in I_k,\quad k\geq0,
	\end{align}
 where
	\begin{align}\label{M:til}
		\til{M}_{L^2}^2(k,t)
	:=\Sob{\eta({\de}_k)}{L^2}^2+C_0
		\frac{{\de}{\mu}^2}{{\kap}}\left[\til{R}_{L^2}^2+
              \bigg(\sup_{s\in I_{k-2}\cup I_{k-1}}\Sob{\eta(s)}{L^2}^2\bigg)\right]{\eric .}
	\end{align}
\end{lem}

\begin{proof}
Suppose $t\in I_k$ for some $k\geq0$.  We perform standard energy
estimates to obtain
	\begin{align}\label{balance:bad}
		\frac{d}{dt}\Sob{\eta}{L^2}^2
	+{{{{\kap}}}}\Sob{\Lam^{\gam/2}\eta}{L^2}^2\leq{\kap}F_{H^{-\gam/2}}^2
	+C\frac{{{{\mu}}}^2}{{{{\kap}}}}
	\left(\Sob{J_h^{{\de}}\tht}{L^2}^2+\Sob{J_h^{{\de}}\eta}{L^2}
                ^2\right).
	\end{align}
Observe that by the Cauchy-Schwarz inequality and \req{Ih:base0} we have
	\begin{align}\notag
		\Sob{J_h^{{\de}}\eta(t)}{L^2}^2\leq C\bigg(\sup_{s\in
                  I_{k-2}\cup I_{k-1}}\Sob{\eta(s)}{L^2}
                ^2\bigg),\quad t\in I_k.
	\end{align}

Returning to \req{balance:bad} and applying these facts along with
\req{tht:bc}, we obtain
	\begin{align}\label{balance:bad2}
		\frac{d}{dt}\Sob{\eta}{L^2}^2+{{{\kap}}}\Sob{\eta}{H^{\gam/2}}^2\leq
                \left({{{\kap}}}F_{H^{-
                    \gam/2}}^2
		+C\frac{{{{\mu}}}^2}{{{{\kap}}}}\Tht_{L^2}^2\right)
		+C\frac{{{{\mu}}}^2}{{{{\kap}}}}\bigg(\sup_{s\in
                  I_{k-2}\cup I_{k-1}}\Sob{\eta(s)}{L^2}^2\bigg).
	\end{align}
Finally, by integrating \req{balance:bad2} {\eric over $[\delta_k,t]$ for}
$t\in I_k$ we arrive
at
	\begin{align}\notag
		\Sob{\eta(t)}{L^2}^2&+{{{\kap}}}
	\int_{{\de}_k}^t\Sob{\eta(s)}{H^{\gam/2}}^2\ ds\\
	&\leq
            \Sob{\eta({\de}_k)}{L^2}^2
	+{\de}\frac{{\mu}^2}{{\kap}}\left[\left(\frac{{\kap}^2}{{\mu}^2}
		F_{H^{-\gam/2}}
              ^2+C\Tht_{L^2}^2\right)+C\bigg(\sup_{s\in
            	I_{k-2}\cup I_{k-1}}\Sob{\eta(s)}{L^2}^2\bigg)\right],
	\end{align}
which can be simplified to \req{bad:Mest} using \req{M2:def:final}, as desired.
\end{proof}

\begin{coro}\label{coro:bad}
Let $k>0$.  Suppose that for each $0\leq \ell\leq k$, there exists $M_\ell>0$ such that
	\begin{align}\notag
		\Sob{\eta(t)}{L^2}^2+\kap\int_{\de_k}^t
			\Sob{\eta(s)}{H^{\gam/2}}^2\ ds\leq
                M_\ell,\quad t\in(-2\de,\de_{\ell+1}].
	\end{align}
Then there exists a constant $C_0>0$, independent of
{\eric $k$},
such that
	\begin{align}\notag
\Sob{\eta(t)}{L^2}^2+\kap\int_{\de_{k+1}}^t\Sob{\eta(s)}{H^{\gam/2}}^2\ ds\leq
M_k+C_0\frac{{\de}{\mu} ^2}{{\kap}}\left(\til{R}_{L^2}^2+M_k\right),\quad
t\in I_{k+1}.
	\end{align}
\end{coro}

{\eric While $\delta$ can be chosen in these bounds
so that the size of $\delta\mu^2/\kappa$ is small,}
this alone does not suffice to obtain uniform-in-time
bounds for $\Sob{\eta(t)}{L^2}$
upon iteration in $k$, which will be crucial in establishing the
synchronization property.  Nevertheless,
these ``rough" bounds {\eric will be useful} in order to
close our estimates {\eric and achieve uniform bounds later}.

\subsubsection{Control of temporal oscillations at fixed spatial scale}
	\label{sect:temp:osc}
We recall from \req{mvt} that we will require estimates for the
time-derivative, $\bdy_t\eta$, but only over length scales $\gtrsim h$, where $h$ measures the spatial resolution of the observables.


\begin{lem}\label{lem:dt}
Let $k>0$.  {\eric Suppose} there exists $M_\ell>0$ such that
	\begin{align}\label{mk:l2}
		\sup_{t\in(-2\de,\de_\ell]}\Sob{\eta(t)}{L^2}\leq M_{\ell-1}
{\eric \qquad\hbox{for each}\qquad 0\leq\ell\leq k+2}.
	\end{align}
Let $c_0>0$ be any constant {\eric such that}
	\begin{align}\label{mu:cond:dt:l2}
		\frac{\mu h^{\gam}}{\kap}\leq \frac{1}{c_0}{\eric .}
	\end{align}
Then there exists a constant $C_0>0$, depending on $c_0$,
but independent of $k$, such that
	\begin{align}\label{dt:l2:past}
		 \Sob{(J_h\bdy_t\eta)(t)}{H^{-\gam/2}}^2
		\leq C_0\left(\frac{2\pi}{h}\right)^2\frac{\kap^2}{h^{\gam}}
		\left(1+\frac{1}{\kap}\frac{1}{h^{2-\gam}}\right)^2
		\left(1+M_k^2+R_{L^2}^2\right)^2
	\end{align}
{\eric holds for all} $t\in (-2\de,\de_{k+1}]$,
and
	\begin{align}\label{dt:l2:now}\notag
		 \Sob{(J_h\bdy_t\eta)(t)}{H^{-\gam/2}}^2\leq&
          C_0\left(\frac{2\pi}{h}\right)^2\kap^2
		\Sob{\eta(t)}{H^{\gam/2}}^2
		+C_0
		\left(\frac{2\pi}{h}\right)^2\frac{M_{k+1}^2}{h^{4-\gam}}
			\Sob{\eta(t)}{L^2}^2
	\\ &\qquad\qquad
		+C_0
		\left(\frac{2\pi}{h}\right)^2\frac{\kap^2}{h^{\gam}}
		\left(M_k^2+R_{L^2}^2\right)
	\end{align}
holds for all $t\in I_{k+1}$.
\end{lem}
\begin{proof}
{\eric By} $(H1)$ we have $\gam/2<1${\eric. Therefore,
by~(\ref{typeI:const11}), see also} \req{typeI:const}, we have
	\begin{align}
		C_I(\gam/2,h)=C\left(\frac{2\pi}h\right).\notag
	\end{align}
{\eric Now, applying} $J_h$ to \req{fb:sqg:prelim},
using the fact that $v$ is {\eric divergence free}, and then taking the
{\eric $H^{-\gamma/2}$-norm} we have
	\begin{align}\label{dt:est:a}\notag
		\Sob{J_h\bdy_t\eta}{H^{-\gam/2}}&\leq
            {{{\kap}}}\Sob{J_h\Lam^\gam\eta}{H^{-\gam/2}}
				+\Sob{J_h\del\cdotp(v\eta)}{H^{-\gam/2}}\\
		&\qquad\qquad+
            \Sob{J_hf}{{\eric H^{-\gamma/2}}}+{{{\mu}}}\Sob{J_h^{{\de}}\eta}{H^{-\gam/2}}
				+{{{\mu}}}\Sob{J_h^\de\tht_j}{H^{-\gam/2}}.
	\end{align}
By $(H1), (H7)$, \req{TI:8}, \req{fgam}, \req{tht:bounds}, and \req{mk:l2}
we may estimate
	\begin{align}
		\kap\Sob{J_h\Lam^\gam\eta(t)}{H^{-\gam/2}}&\leq
                C \left(\frac{2\pi}h\right)\kap h^{-\gam/2}\Sob{\eta(t)}{L^2}\notag\\
		&\leq C\left(\frac{2\pi}h\right)\kap h^{-\gam/2} M_k,\quad t\in(-2\de,\de_{k+1}],\notag\\
                \kap\Sob{J_h\Lam^\gam\eta(t)}{H^{-\gam/2}}&\leq C
                 \left(\frac{2\pi}h\right)\kap \Sob{\eta(t)}{H^{\gam/2}},\quad t\in I_{k+1}\notag\\
		\Sob{J_hf}{H^{-\gam/2}}&\leq C\left(\frac{2\pi}h\right)\kap F_{H^{-\gam/2}},\notag\\
        {{{\mu}}}\Sob{J^{{\de}}_h\eta(t)}{H^{-\gam/2}}&\leq
                C\left(\frac{2\pi}h\right)\mu h^{\gam/2}\bigg(\sup_{s\in(-2\de,\de_{k+1}]}\Sob{\eta(s)}{L^2}\bigg)\notag\\
		&\leq C{\left(\frac{2\pi}h\right)}\mu h^{\gam/2}M_k,\quad t\in(-2\de, \de_{k+2}],\notag\\
	\mu\Sob{J_h^{\de}\tht(t)}{H^{-\gam/2}}&\leq C\left(\frac{2\pi}h\right)\mu h^{\gam/2}\Tht_{L^2},\quad t>-2\de.\notag
	\end{align}
For the quadratic term {\eric apply
\req{TI:9}, the Cauchy--Schwarz} inequality {\eric and} the
fact that~$\Ri^\perp$ is a
bounded operator in $L^2$ to estimate
	\begin{align}\notag
		\Sob{J_h\del\cdotp(v\eta)}{H^{-\gam/2}}\leq C
                    {\left(\frac{2\pi}h\right)}h^{-2+\gam/2}\Sob{v\eta}{L^1}\leq
                    C{\left(\frac{2\pi}h\right)}h^{-2+\gam/2}M_k^2,\quad t\in(-2\de,\de_{k+1}],
	\end{align}
and
	\begin{align}\notag
		\Sob{J_h\del\cdotp(v\eta)}{H^{-\gam/2}}\leq C{\left(\frac{2\pi}h\right)}h^{-2+\gam/2}M_{k+1}\Sob{\eta(t)}{L^2},\quad t\in I_{k+1}.
	\end{align}
	
Upon collecting these estimates, returning to \req{dt:est:a}, we apply \req{tht:bounds}, and \req{M2:def:final} to obtain
	\begin{align}
		 \Sob{(J_h\bdy_t\eta)(t)}{H^{-\gam/2}}\leq C
             {\left(\frac{2\pi}h\right)}\left(1+\frac{\mu h^\gam}{\kap}\right)^2\frac{\kap}{h^{\gam/2}}\left(1+\frac{1}{\kap}\frac{1}{h^{2-\gam}}\right)\left(1+M_k+R_{L^2}\right)^2,\notag
	\end{align}
for $t\in(-2\de, \de_{k+1}]$, as well as
	\begin{align}
		\Sob{(J_h\bdy_t\eta)(t)}{H^{-\gam/2}}\notag
		&\leq
             C{\left(\frac{2\pi}h\right)}\kap\Sob{\eta(t)}{H^{\gam/2}}
+C{\left(\frac{2\pi}h\right)}\frac{M_{k+1}}{h^{2-\gam/2}}
			\Sob{\eta(t)}{L^2}
\\ &\qquad
	+C{\left(\frac{2\pi}h\right)}\frac{\kap}{h^{\gam/2}}
		\left(1+\frac{\mu h^{\gam}}{\kap}\right)^2
		\left(M_k+R_{L^2}\right),\notag
	\end{align}
for $t\in I_{k+1}$.  
{\eric Note that in collecting the terms we have used the fact that all
constants and variables have been non-dimensionalized so that, for
example, terms such as $1+1/(\kappa h^{2-\gamma})$
and $1+M_k+R_{L^2}$ make sense.}
Thus, upon squaring both sides of
these inequalities, then applying Young's inequality and \req{mu:cond:dt:l2}, we arrive at \req{dt:l2:past} and \req{dt:l2:now}.
\end{proof}

\subsubsection{Growth during initial transient period}\label{sect:l2:trans}
Due to the delay, we must quantify bounds over the initial transient
period during which the feedback effects from large scales can amplify the solution.
{\eric Consider the definition of $\Gamma_{2,k}$ for $k=-1,0,1,\ldots$ given
by (\ref{Gamk}).}
Observe that
	\begin{align}\label{rmk:gam}
		\Gam_{2,k-1}\leq\Gam_{2,k},\quad k\geq0.
	\end{align}
By \req{gam:bounds}, Lemma \ref{prop:bad}, and Corollary
\ref{coro:bad} we have
	\begin{align}\notag
		\Sob{\eta(t)}{L^2}^2+{{\kap}}
			\int_{\de_k}^t\Sob{\eta(s)}{H^{\gam/2}}^2\ ds\leq
                {\eric \Gam_{2,k} ^2},\quad t\in I_{k},\quad k=-1,0,1.
	\end{align}
It then follows from \req{rmk:gam} that
	\begin{align}\label{M:til1}
		\Sob{\eta(t)}{L^2}^2+{{\kap}}
			\int_{\de_k}^t\Sob{\eta(s)}{H^{\gam/2}}^2\ ds\leq
                {\eric \Gam_{2,1} ^2\leq\Gam_{2,1}^2}+\rho,
                \quad t\in I_k,\quad k=-1,0,1,
	\end{align}
for any $\rho\geq0$.

As we will see, the choice of $\rho$ will be dictated by the estimates
{\eric \req{l2:prop:est} and \req{Mtil:k} below.}
In anticipation of this, {\eric consider the third definition
of (\ref{M2:def:final}) given by}
	\begin{align}\label{M2:def}
		M_{L^2}^2:=\Gam_{2,1}^2+8R_{L^2}^2.
	\end{align}
Then \req{M:til1} implies
	\begin{align}\label{M2:avg}
		\Sob{\eta(t)}{L^2}^2+\frac{{\kap}}2
			\int_{\de_k}^t\Sob{\eta(s)}{H^{\gam/2}}^2\ ds\leq
                M_{L^2} ^2,\quad t\in I_k,\quad k=-1,0,1.
	\end{align}
Therefore, the {\eric conclusion} of Proposition \ref{prop:l2:avg} is that there is a
choice of $\rho$ such that the bound {\eric given by}
\req{M2:avg} propagates beyond the initial transient period,
{\eric provided that}
${\de}$ is chosen small enough.  In
particular, Proposition \ref{prop:l2:avg} provides a more precise version of
\req{M2:avg}, which not only allows this bound to propagate through
all times $t>2{\de}$, but {\eric in such a way that it}
eventually ``forgets" the initializing
function, $g$, as well.

We are {\eric now} ready to prove Proposition \ref{prop:l2:avg}.

\subsubsection{Proof of Proposition \ref{prop:l2:avg}}\label{sect:l2:avg}

We proceed by induction on {\eric $k$}.
{\eric As we shall see shortly}, by
Lemma~\ref{gron}~(ii), it suffices to show {\eric for $k\ge 2$
and $t\in I_k$} that
	\begin{align}\label{indbound}
		\Sob{\eta(t)}{L^2}^2&+\frac{{{\kap}}}2
			\int_{{{\de}}_{k}}^te^{-({{{\mu}}}/2)(t-s)}\Sob{\eta(t)}
            {H^{\gam/2}}^2\ ds\notag\\&\leq\left(\Sob{\eta({{\de}}_k)}{L^2}^2
				-8{R_{L^2}^2}\right)e^{-({{{\mu}}}/2)(t-{{\de}}_{k})}
                    +8{R_{L^2}^2}.
	\end{align}
We proceed in {\eric three steps}.
Step~I proves the base case when $k=2$
while Step~II provides the induction step thereby completing
the induction.
Finally, Step~III
uses (\ref{indbound}) along with Lemma~\ref{gron}~(ii) to obtain
(\ref{m2:full}) and (\ref{mgam2:avg}) which finishes the proof.

{\eric {\flushleft\textbf{I. Base case.}}
Let $k=2$ and suppose $t\in I_2$.
By Corollary \ref{coro:bad} and \req{M2:avg}} we have
	\begin{align}
		\Sob{\eta(t)}{L^2}^2 \leq
                \Gam_{2,1}^2+C\frac{{\de}{\mu}^2}{{\kap}}\left(\Gam_{2,1}^2+\til{R}_{L^2}
                ^2\right)=\Gam_{2,2}^2,\quad t\in I_2.\notag
	\end{align}
It then follows from \req{M2:avg} and the second condition of
\req{de:small:l2} that
	\begin{align}\label{bound1}
		\Sob{\eta(t)}{L^2}^2\leq\Gam_{2,2}^2\leq
                M_{L^2}^2,\quad t\in(-2{\de}, 3{\de}].
	\end{align}
%
{\eric Multiply} \req{fb:k} by $\eta$, integrate over $\T^2$, and apply
\req{mvt} to obtain
$$
	\frac{1}2\frac{d}{dt}\Sob{\eta}{L^2}^2
	+{{\kap}}\Sob{\Lam^{\gam/2}\eta}{L^2}^2+{{{\mu}}}
          \Sob{\eta}{L^2}^2 = {\eric \I_1+\I_2+\I_3+\I_4}
$$
{\eric where
$$
	\I_1= \int f\eta\ dx,\qquad
	\I_2={{{\mu}}}\int
          (\eta-J_h\eta)\eta\ dx,\qquad
	\I_4= {{{\mu}}}\int (J_h^\de\tht)
          \eta\ dx
$$
and
$$
	\I_3= \frac{\mu}{\de}
          \int\int_{t-2\de}^{t-\de}\int_s^t[
          J_h\bdy_ \tau\eta(\tau)]\eta(t)\ d\tau dsdx.
$$}
%
Observe that by \req{TI},
Cauchy-Schwarz inequality,  Young's inequality, and \req{tht:bc} we have
\begin{align*}
		\I_1&\leq\frac{1}{{{{\kap}}}}\Sob{\Lam^{-\gam/2}f}{L^2}^2+\frac{{{{\kap}}}}{4}\Sob{\Lam^{\gam/
                    2}\eta}{L^2}^2,\\
                 \I_2 &\leq C{{{\mu}}}h^{\gam/2}\Sob{\Lam^{\gam/2}\eta}{L^2}\Sob{\eta}{L^2}\leq
                \frac{{{{\kap}}}}8\Sob{\Lam^{\gam/2}\eta}{L^2}^2+Ch^{\gam}\frac{{{{\mu}}}^2}{{{{\kap}}}}
                \Sob{\eta}{L^2}^2
\intertext{and}
                \I_4&\leq
                    {{{\mu}}}\Sob{J_h^{{\de}}\tht}{L^2}\Sob{\eta}{L^2}\leq
                    C{{{{\mu}}}}\Tht_{L^2}^2+\frac{{{{\mu}}}}{4}\Sob{\eta}{L^2}^2.
\end{align*}
{\eric Further estimating $\I_1, \I_2$ and $\I_4$ using}
\req{m:mu:avg:l2}, \req{tht:bc} and \req{M2:def:final} gives
	\begin{align}\label{balance:bc}
		\frac{1}2\frac{d}{dt}\Sob{\eta}{L^2}^2+\frac{5{{{\kap}}}}{8}\Sob{\Lam^{\gam/2}\eta}{L^2}^2+
                \frac{{{{\mu}}}}2\Sob{\eta}{L^2}^2\leq{\mu}R_{L^2}^2+\I_3.
	\end{align}
{\eric To estimate $\I_3$,} apply Fubini's theorem, Parseval's theorem, the Cauchy-Schwarz inequality, \req{Ih:base0}, and Young
inequalities {\eric in the} following sequence
of estimates
	\begin{align}\label{mvt:bc}
		\I_3&\leq \frac{\mu}{\de}\int_{t-2{{\de}}}^{t-{{\de}}}\int_s^t\left|\int
                (J_h\bdy_\tau
                \eta(x,\tau))(\eta(x,t))\ dx\right|d\tau\ ds\notag\\ &\leq
                C{{{\mu}}}\frac{1}{{{\de}}}\int_{t-2{{\de}}}^{t-{{\de}}}
			{\eric (t-s)}
		\left(\int_s^t\Sob{J_h\bdy_\tau\eta(\tau)}{H^{-\gam/2}}^2\ d\tau\right)^{1/2}\ ds\Sob{\eta(t)}{H^{\gam/2}}\notag\\
                  &\leq
                C{{{\mu}}}\left(\frac{1}{{{\de}}}\int_{t-2{{\de}}}^{t-{{\de}}}(t-s)
                \int_{t-2{{\de}}}^t\Sob{J_h\bdy_\tau\eta(\tau)}{H^{-\gam/2}}^2\ d\tau\ ds\right)^{1/2}\Sob{\eta(t)}
                    {H^{\gam/2}}\notag\\
                    &\leq
                    \frac{1}2{\left({C}\frac{\de{{{\mu}^2}}}{{{{\kap}}}}\int_{t-2{{\de}}}^t\Sob{J_h\bdy_s\eta(s)}{H^{-\gam/2}}^2\ ds\right)}+\frac{{{{\kap}}}}8\Sob{\eta(t)}
                         {H^{\gam/2}}^2.
	\end{align}
Let
	\begin{align}\label{*:L2}
		\mystara(t):={{C}\frac{\de{{{\mu}^2}}}{{{{\kap}}}}
	\int_{t-2{{\de}}}^t\Sob{J_h\bdy_s\eta(s)}{H^{-\gam/2}}^2\ ds}.
	\end{align}
Observe that $\mystara(t)=\mystara_0+\mystara_1+\mystara_2(t)$,
where for $\ell\geq0$, we {\eric have defined}
	\begin{align}\label{star:kt}
		\mystara_\ell(t):={C{{\de}}\frac{{{{\mu}^2}}}{{{{\kap}}}}\int_{{\de}_\ell}^{t}
            \Sob{J_h\bdy_s\eta(s)}{H^{-\gam/2}}^2\ ds}
\qquad\hbox{and}\qquad
	\mystara_\ell:=\mystara_\ell(\delta_{\ell+1}).
	\end{align}
Returning to \req{balance:bc} and applying \req{mvt:bc} and \req{m:mu:avg:l2}, we have
	\begin{align}\label{balance:bc2}
	\frac{d}{dt}\Sob{\eta(t)}{L^2}^2
		+{{{{\kap}}}}\Sob{\eta(t)}{H^{\gam/2}}^2
		+{{{{\mu}}}}\Sob{\eta(t)} {L^2}^2
	&\leq
    \frac{2}{c_0}\frac{\kap}{h^{\gam}}{R_{L^2}^2}
		+\mystara_0+\mystara_1+\mystara_2(t).
	\end{align}
{\eric To obtain bounds on $\mystara_0$ and $\mystara_1$ define
\begin{align}\label{Bm}
	O_1({\de}^2):=
		C{\de}^2\frac{\kap^3}{h^{3\gam}}\left(
			\frac{2\pi}{h}\right)^2\left(1
				+\frac{1}{\kap}\frac{1}{h^{2-\gam}}\right)^2
			\left(1+M_{L^2}^2+R_{L^2}^2\right)^2
\end{align}
so that, upon simplifying \req{dt:l2:past} with \req{m:mu:avg:l2}, we
obtain from Lemma \ref{lem:dt} and \req{bound1} that}
	\begin{align}\label{dt:I01}
		{\eric \max\{\mystara_0, \mystara_1\}} \leq O_1({\de}^2).
	\end{align}
{\eric To bound $\mystara_2(t)$ for} $t\in I_2$,
observe that by Lemma \ref{lem:dt} and \req{bound1} we have
	\begin{align}\label{dt:I2}
		\mystara_2(t)\leq
          C_1(h){\de}\int_{{\de}_2}^t\Sob{\eta(s)}{L^2}^2\ ds+C_2(h){\de}\int_{{\de}_2}^t
          \Sob{\eta(s)}{H^{\gam/2}}^2\ ds+O_2({\de}^2).
	\end{align}
where, upon simplifying \req{dt:l2:now} with \req{m:mu:avg:l2},
we {\eric have defined}
\begin{align*}
	C_1(h):=
		C\kap\left(\frac{2\pi}{h}\right)^2\frac{M_{L^2}^2}{h^{4+\gam}},
\qquad
   C_2(h):=C\frac{\kap^3}{h^{2\gam}}\left(\frac{2\pi}{h}\right)^2
\end{align*}
{\eric and}
$$
                O_2({\de}^2):=C\de^2\frac{\kap^3}{h^{3\gam}}\left(\frac{2\pi}{h}\right)^2\left(M_{L^2}^2+R_{L^2}^2\right).
$$
Combining \req{dt:I01} and \req{dt:I2} then gives
	\begin{align}\label{dt:est}
		\mystara(t)\leq
          C_1(h){\de}\int_{{\de}_2}^t\Sob{\eta(s)}{H^{\gam/2}}^2\ ds+C_2(h){\de}\int_{{\de}_2}^t
          \Sob{\eta(s)}{L^2}^2\ ds+O_1({\de}^2)+O_2({\de}^2).
	\end{align}
Observe that since $O_2({\de}^2)\leq O_1({\de}^2)$, it follows from
the third condition on $\de$ in \req{de:small:l2} that
	\begin{align}\notag
		O_1({\de}^2)+O_2({\de}^2)\leq \frac{2}{c_0}\frac{\kap}{h^\gam}R_{L^2}^2.
	\end{align}
Thus, upon returning to \req{mvt:bc}, we have
	\begin{align}\notag
		{\eric I_3}\leq
                C_1(h){\de}\int_{{\de}_2}^t\Sob{\eta(s)}{H^{\gam/2}}^2\ ds+C_2(h){\de}\int_{{\de}_2}^t
                \Sob{\eta(s)}{L^2}^2\ ds+\frac{2}{c_0}\frac{\kap}{h^\gam}
                R_{L^2}^2+\frac{{{{\kap}}}}8\Sob{\eta(t)}{H^{\gam/2}}^2.
	\end{align}
By applying {\eric the resulting bounds on $\mystara(t)$ in \req{balance:bc2}}, we have
{\eric for $t\in I_2$ that}
	\begin{align}\label{balance:bc4}
		\frac{d}{dt}\Sob{\eta}{L^2}^2&+{{{{\mu}}}}\Sob{\eta}{L^2}^2+{{{{\kap}}}}\Sob{\eta}{H^{\gam/2}}^2\notag\\
	&\leq
                \frac{4}{c_0}\frac{\kap}{h^\gam}{R_{L^2}^2}+C_1(h){\de}\int_{{\de}_2}^t\Sob{\eta(s)}{L^2}^2\ ds+C_2(h){\de}\int_{{\de}_2}^t
                \Sob{\eta(s)}{H^{\gam/2}}^2\ ds.
	\end{align}
	
Now observe that \req{gron:cond:l2} ensures that \req{de:cond:gron} holds in Lemma \ref{gron} with
	\[
		a=\mu, \quad b={{{\kap}}}, \quad A=C_1, \quad B=C_2,
                \quad F=\frac{4}{c_0}{{{\frac{\kap}{h^\gam}}}}{R_{L^2}^2}.
	\]
Applying Lemma \ref{gron} (i) then gives
	\begin{align}\label{l2:prop:est}
		\Sob{\eta(t)}{L^2}^2&+\frac{{\kap}}2\int_{{\de}_2}^te^{-({\mu}/2)(t-s)}\Sob{\eta(s)}{H^{\gam/2}}
                ^2\ ds\notag\\&\leq
                \left(\Sob{\eta({\de}_2)}{L^2}^2-8{R_{L^2}^2}\right)e^{-({{{\mu}}}/2)(t-{\de}_2)}+8{R_{L^2}^2},
                \quad t\in I_2,
	\end{align}
{\eric which finishes the proof of the base case.}

{\eric {\flushleft{\textbf{II. Induction Step.}}} Suppose
$k\geq2$ and} for each $\ell=2,\dots, k$ {\eric and $t\in I_\ell$ that}
	\begin{align}\label{induct:hyp}
		\Sob{\eta(t)}{L^2}^2+\frac{{{\kap}}}2\int_{{\de}_{\ell}}^te^{-({{{\mu}}}/2)(t-s)}\Sob{\eta(s)}
                    {H^{\gam/2}}^2\ ds\leq\left(\Sob{\eta({\de}_\ell)}{L^2}^2-8{R_{L^2}^2}\right)e^{-({{{\mu}}}/2)(t-{\de}_\ell)}
                    +8{R_{L^2}^2}
	\end{align}
We show the bound corresponding {\eric to $\ell=k+1$} holds
{\eric for $t\in I_{k+1}$.}

{\eric As already demonstrated, our choice of $\delta$ has been chosen
so that the hypotheses of Lemma~\ref{gron} hold for the differential
inequality (\ref{balance:bc4}).  These hypotheses are also satisfied
for the modified inequality obtained by replacing $\delta_2$ by $\delta_\ell$
for $\ell=2,\ldots,k$
which we write as
\begin{align}\label{balance:bcl}
		\frac{d}{dt}\Sob{\eta}{L^2}^2&+{{{{\mu}}}}\Sob{\eta}{L^2}^2+{{{{\kap}}}}\Sob{\eta}{H^{\gam/2}}^2\notag\\
	&\leq
                \frac{4}{c_0}\frac{\kap}{h^\gam}{R_{L^2}^2}+C_1(h){\de}\int_{{\de}_\ell}^t\Sob{\eta(s)}{L^2}^2\ ds+C_2(h){\de}\int_{{\de}_\ell}^t
                \Sob{\eta(s)}{H^{\gam/2}}^2\ ds
	\end{align}
for $t\in I_\ell$.
Now,
dropping the integral in (\ref{induct:hyp}) and rewriting the
last term yields
$$
 \|\eta(t)\|_{L^2}^2
\le \|\eta(\delta_\ell)\|_{L^2}^2 e^{-(\mu/2)(t-\delta_\ell)}
	+8R_{L^2}^2\int_{\delta_\ell}^{t} {2\over\mu}e^{-(\mu/2)(t-s)} ds
\quad\hbox{for}\quad t\in I_\ell
$$
so that by iterating part (ii) of Lemma~\ref{gron}
for $\ell=2,\ldots,k$ we obtain
$$
 \|\eta(t)\|_{L^2}^2
\le \|\eta(\delta_2)\|_{L^2}^2 e^{-(\mu/2)(t-\delta_2)}
	+8R_{L^2}^2\int_{\delta_2}^{t} {2\over\mu}e^{-(\mu/2)(t-s)} ds
\quad\hbox{for}\quad t\in (\delta_2,\delta_{k+1}].
$$
Since $\|\eta(\delta_2)\|_{L^2}^2 \le \Gamma_{2,1}^2$
by (\ref{M:til1}), we immediately obtain (\ref{m2:full}) and
in particular that
	\begin{align}\label{l2:ind:apriori}
		\Sob{\eta(t)}{L^2}^2\leq \Gam_{2,1}^2+8{R_{L^2}^2}=M_{L^2}^2,
                \quad t\in ({\de}_{2},{\de}_{k+1}].
	\end{align}
By Corollary \ref{coro:bad} it follows that
	\begin{align}\notag
		\Sob{\eta(t)}{L^2}^2\leq
               M_{L^2}^2+C\frac{{\de}{\mu}^2}{{\kap}}
                \left(M_{L^2}^2+\til{R}_{L^2}^2\right),\quad
                t\in I_{k+1}.
	\end{align}
Thus, by the second condition in \req{de:small:l2} we have
	\begin{align}\label{Mtil:k}
		\Sob{\eta(t)}{L^2}^2+\frac{\kap}4\int_{\de_{k+1}}^t\Sob{\eta(s)}{H^{\gam/2}}\ ds\leq
                M_{L^2} ^2,\quad t\in I_{k+1}.
	\end{align}}

Now proceed exactly as in the base case, this time making use of the bounds
\req{l2:ind:apriori} and \req{Mtil:k}.
Indeed, we may derive \req{balance:bc2} as before.
Then, since $t\in I_{k+1}$, we may split the time integral over
three regions:
	\begin{align}\notag
		\int_{t-2\de}^{{t}}\leq
                \int_{I_{k-1}}+\int_{I_k}+\int_{{\de}_{k+1}}^t.
	\end{align}

Over $I_{k-1}$ and $I_k$, Lemma \ref{lem:dt} and \req{Mtil:k} imply
\req{dt:I01} for $\mystara_{k-1}$ and $\mystara_k$.
Over $I_{k+1}$, we have \req{Mtil:k}, so that Lemma \ref{lem:dt}
implies \req{dt:I2} for $\mystara_{k+1}(t)$.
We then deduce \req{dt:est} for $t\in I_{k+1}$, which leads to the
differential inequality \req{balance:bcl} with $\ell=k+1$.
Applying Lemma~\ref{gron} (i) as before then
yields
	\begin{align}\label{Ikp1:l2}
		\Sob{\eta(t)}{L^2}^2&+\frac{{{\kap}}}2\int_{{\de}_{k+1}}^te^{-({{{\mu}}}/2)(t-s)}\Sob{\eta(s)}
                    {H^{\gam/2}}^2\ ds\notag\\&\leq\left(\Sob{\eta({\de}_{k+1})}{L^2}^2-8{R_{L^2}^2}\right)e^{-({{{\mu}}}/2)(t-{\de}_{k
                        +1})}+8{R_{L^2}^2}
	\end{align}
{for $t\in I_{k+1}$ thus completing the induction.}

{\eric \flushleft{\bf III.  Finish the Proof.}
We have already obtained (\ref{m2:full}) for all values of $k$
by iterating Lemma~\ref{gron} (ii)
as part of the induction step.
To obtain (\ref{mgam2:avg}) drop the first term in (\ref{indbound})
and keep the integral.  Consequently,
we may then deduce that
	\begin{align}
		\frac{{{\kap}}}2
                e^{-({{{\mu}}}/2)(t-{\de}_k)}\int_{{\de}_k}^t\Sob{\eta(s)}{H^{\gam/2}}^2\ ds\leq
                \left(\Sob{\eta({\de}_k)}{L^2}^2-8{R_{L^2}^2}\right)e^{-({{{\mu}}}/2)(t-{\de}_k)}+8{R_{L^2}^2},\quad
                t\in I_k.\notag
	\end{align}
Since the first condition in \req{de:small:l2} and \req{m:mu:avg:l2} together
imply $e^{(\mu/2)\de}\leq2$, it follows from \req{M2:def:final} and \req{M:til1} that
	\begin{align}\label{m2gam2}
		\frac{{{\kap}}}4
                \int_{I_k}\Sob{\eta(s)}{H^{\gam/2}}^2\ ds&\leq
                \Sob{\eta({\de}_k)}{L^2}
                ^2+8{R_{L^2}^2}(e^{({{{\mu}}}/2){\de}}-1)\leq
                M_{L^2}^2.
	\end{align}
This completes the proof.
}\qed

\begin{rmk}\label{rmk:l2vs}
We point out that the energy estimates in $L^p$ and $H^\s$ will not proceed along these lines, the reason being that even if one were to do so, the resulting bounds would still not be independent of $h$.  So long as these bounds are uniform-in-time, however, we will be able to use them strengthen the topology of convergence in which the synchronization takes place via interpolation.  We will thus be content with rather modest bounds in $L^p$ and $H^\s$.
\end{rmk}

\subsection{$L^2$ to $L^p$ uniform bounds}

We will prove the following ``good" bound:

\begin{prop}\label{prop:lp:avg}
Let $F_{L^p}, \Tht_{L^p}, M_{L^2}$ be given by \req{fp},
\req{tht:bounds}, \req{M2:def:final}, respectively.  Define
	\begin{align}\label{Rp:til}
		{{\til{R}_{L^p}}}^p(h):=F_{L^p}^p+\Tht_{L^p}^p+\til{C}(h,p)^pM_{L^2}^p,
	\end{align}
where
	\begin{align}\label{Ctil}
		\til{C}(h,p)^p:=1+h^{-(p-2)}.
	\end{align}
Let $c_0>0$ be any constant.  Suppose that
	\begin{align}\label{mu:hgam:lp}
		\frac{\mu h^\gam}{\kap}\leq \frac{1}{c_0}.
	\end{align}
Then there exists a constant $C_0>0$, depending on $c_0$, such that
	\begin{align}\notag
		\Sob{\eta(t)}{L^p}^p\leq
                \left(\Gam_{L^p}^p-\left(\frac{C_0}{h^{\gam}}\right)^p{{\til{R}_{L^p}^p}}
                \right)e^{-{{{\kap}}}
                  t}+\left(\frac{C_0}{h^{\gam}}\right)^p{{\til{R}_{L^p}}}^p,
                \quad t\geq0,
	\end{align}
In particular,
	\begin{align}\notag
		\Sob{\eta(t)}{L^p}\leq \til{M}_{L^p},\quad t\geq0,
	\end{align}
where
	\begin{align}\label{Mp:def}
		\til{M}_{L^p}(h)^p:=\Gam_{L^p}^p+\left(\frac{C_0}{h^{\gam}}\right)^p\til{R}_{L^p}
                ^p(h).
	\end{align}
\end{prop}

\begin{proof}
Observe that by \req{gam:bounds}, we have
	\begin{align}\notag
		\Sob{\eta(t)}{L^p}\leq \Gam_{L^p},\quad t\in I_{-1}.
	\end{align}
For $t\geq0$, the evolution of $\Sob{\eta(t)}{L^p}$ is obtained by
multiplying \req{fb:sqg:avg} by $\eta| \eta|^{p-2}$, integrating over
$\T^2$, applying {Proposition} \ref{lb}, H\"older's inequality, Young's
inequality, and \req{fp} to obtain
	\begin{align}\label{balance:lp}
		\frac{1}p\frac{d}{dt}\Sob{\eta}{L^p}^p+\frac{2\kap}p\Sob{\Lam^{\gam/2}(\eta^{p/2})}{L^2}
                ^2\leq{C^p}\frac{\kap}pF_{L^p}^p+C^p\frac{\kap}p\left(\frac{\mu}{\kap}\right)^p\left(\Sob{J_h^\de\eta}
                     {L^p}^p+\Sob{J_h^\de\tht}{L^p}^p\right)+\frac{\kap}{2p}\Sob{\eta}{L^p}^p
	\end{align}
Applying H\"older's inequality, the Fubini-Tonelli theorem, \req{Ih:base1}
{\eric with $q=2$}, and Proposition \ref{prop:l2:avg} we have
	\begin{align}\label{Sm}
		\Sob{J_h^\de\eta}{L^p}^p
		\leq
			{{1\over\delta}\int_{t-2\delta}^{t-\delta}
				\Sob{J_h\eta(s)}{L^p}^p ds }
		\leq
                C^ph^{2-p}M_{L^2}^p.
		\end{align}
Similarly, by \req{tht:bc}, $\Sob{J_h^\de\tht}{L^p}^p\leq {C_J^p}\Tht_{L^p}^p $.
{Upon defining
$$
	\langle \eta^{p/2}\rangle_{\T^2} = {1\over 4\pi^2} \int_{\T^2}
			\eta^{p/2}dx
$$
observe} that
	\begin{align}\label{poin:avg}
		\Sob{\eta}{L^p}^p-(4\pi^2)^{-1}\Sob{\eta}{L^{p/2}}^p=\Sob{\eta^{p/2}-
	{\langle \eta^{p/2}\rangle_{\T^2}}}{L^2}
                ^2\leq C(2\pi)^{\gam}\Sob{\Lam^{\gam/2}(\eta^{p/2})}{L^2}^2.
	\end{align}
{Note that the constant $(2\pi)^\gam$ carries the units of $L^{\gam}$; however, as $L=2\pi$ throughout this paper we avoid keeping
track of the dimensions in this case, and simply denote the prefactor $C(2\pi)^\gam$ by $C$.}
By interpolation, Young's inequality, and H\"older's inequality we
have
	\begin{align}\label{p2:inter}
		\Sob{\eta}{L^{p/2}}^p\leq
                \Sob{\eta}{L^{p}}^{\frac{p(p-2)}{p-1}}\Sob{\eta}{L^1}^{\frac{p}{p-1}}
                \leq
                C^p\left(\frac{p-2}{p-1}\right)^{p-2}M_{L^2}^{p}+{\pi^2}\Sob{\eta}{L^p}^p.
	\end{align}
Upon combining \req{Sm}, \req{poin:avg}, \req{p2:inter}, \req{Ctil}
and returning to \req{balance:lp}, we arrive at
	\begin{align}\notag
		\frac{d}{dt}\Sob{\eta}{L^p}^p+\kap\Sob{\eta}{L^p}^p\leq
                C^p\frac{\kap}{h^{\gam p}}\left(\frac{\mu h^\gam}{\kap}\right)^p
                \left(F_{L^p}^p+\til{C}(h,p)^pM_{L^2}^p+\Tht_{L^p}^p\right).
	\end{align}
An application of \req{mu:hgam:lp} and Gronwall's inequality completes the proof.
\end{proof}

\subsection{Uniform $H^\s$-estimates}\label{sect:hs:avg}
As in the previous section, we obtain ``good" $H^\s$-bounds without appealing to time-derivative estimates. 

\begin{prop}\label{prop:bad:hs}
Let $M_{L^2}$ be given by
\req{M2:def} {\eric and} let $\Tht_{H^\s}, \til{M}_{L^p}$ be given by \req{sqg:hs:ball},
\req{Mp:def}, respectively.  Define
	\begin{align}\label{xi}
		\til{\Xi}_{L^p}(h):=\left(\frac{\til{M}_{L^p}(h)}{{\kap}}\right)^{\frac{\s}{\gam-1-2/p}},
	\end{align}
as well as
	\begin{align}\label{Rtilsig}
		F_{H^{\s-\gam/2}}:=\frac{1}{{\kap}}\Sob{f}{H^{\s-\gam/2}}
                {\eric \quad\hbox{and}\quad}
 R_{H^\s}^2:=F_{H^{\s-\gam/2}}^2+\Tht_{L^2}^2.
	\end{align}
Let $c_0>0$ be the {\eric constant given in Proposition \ref{prop:l2:avg}.}
Suppose that
	\begin{align}\label{mu:hgam:hs}
		\frac{\mu h^\gam}{\kap}\leq \frac{1}{c_0}.
	\end{align}
Then there exists a
constant $C_0>0$, depending on $c_0$, such that
	\begin{align}\notag
		\Sob{\eta(t)}{H^\s}^2
			&\leq \Gam_{H^\s}^2e^{-\kap t}+C_0\left[\left(\til{\Xi}_{L^p}^{\frac{2\s+\gam}\s}+\frac{1}{h^{2\gam}}\right)M_{L^2}^2+\frac{1}{h^{2\gam}}R_{H^\s}^2\right]\left(1-e^{-\kap t}\right),\notag
	\end{align}
holds for $t\geq0$ and $\s\leq\gam/2$, and
	\begin{align}\notag
		\Sob{\eta(t)}{H^\s}^2
			&\leq \Gam_{H^\s}^2e^{-\kap t}+C_0\left[\left(\til{\Xi}_{L^p}^{\frac{2\s+\gam}\s}+\frac{1}{h^{2\s+\gam}}\right)M_{L^2}^2+\frac{1}{h^{2\s+\gam}}+\frac{1}{h^{2\s+\gam}}R_{H^\s}^2\right]\left(1-e^{-\kap t}\right),\notag
	\end{align}
holds for $t\geq0$ and $\s>\gam/2$.
\end{prop}

\begin{proof}
Suppose $t\geq0$.  We multiply \req{fb:sqg:avg} by
$\Lam^{2\s}\eta$ and integrate over $\T^2$ to obtain
	\begin{align}\label{balance:hs}
		\frac{1}2\frac{d}{dt}&\Sob{\eta}{H^\s}^2
		+{{{\kap}}}\Sob{\eta}{H^{\s+\gam/2}}^2\notag\\
                &=-\int {v}\cdotp{\del}\eta\Lam^{2\s}\eta\ dx+\int
                f\Lam^{2\s}\eta\ dx+{{{\mu}}} \int
                J_h^{{\de}}\eta\Lam^{2\s}\eta\ dx+{{{\mu}}}\int
                J_h^{{\de}}\tht\Lam^{2\s}\eta\ dx\notag\\ &=
				\J_1 + \J_2 + \J_3 + \J_4.
	\end{align}
We estimate $\J_1$ with H\"older's inequality, interpolation, and Young's inequality as in
\cite{const:wu:qgwksol, jmt}, and invoke \req{xi} to obtain
	\begin{align}
		|\J_1|&\leq
                C\Sob{\Lam^{\s+\gam/2}\eta}{L^2}^{\frac{2\s-(\gam-1-2/p)}{\s}}\Sob{\Lam^{\gam/
                    2}\eta}{L^2}^{\frac{\gam-1-2/p}\s}\Sob{\eta}{L^{p}}\notag
                    	\leq
                \frac{{{{\kap}}}}{\eric 10}\Sob{\eta}{H^{\s+\gam/2}}^2+
                     {C}\til{\Xi}_{L^p}^2({{{\kap}}}\Sob{\eta}{H^{\gam/2}}^2).\notag
	\end{align}
Note that {$(H1)$, $(H2)$ and $(H3)$ are needed
for} the interpolation.  We interpolate once more to obtain
	\begin{align}\notag
		\Sob{\eta}{H^{\gam/2}}\leq C\Sob{\eta}{H^{\s+\gam/2}}^{\frac{\gam/2}{\s+\gam/2}}\Sob{\eta}{L^2}^{\frac{\s}{\s+\gam/2}}.
	\end{align}
Thus, by Young's inequality, we have
	\begin{align}
		C\til{\Xi}_{L^p}^2(\kap\Sob{\eta}{H^{\gam/2}}^2)\leq C\kap \Sob{\eta}{H^{\s+\gam/2}}^{\frac{\gam}{\s+\gam/2}}(\til{\Xi}_{L^p}^2\Sob{\eta}{L^2}^{\frac{2\s}{\s+\gam/2}})\leq \frac{\kap}{\eric 10}\Sob{\eta}{H^{\s+\gam/2}}^2+C\kap\til{\Xi}_{L^p}^{2+\gam/\s}M_{L^2}^2.\notag
	\end{align}

For
$\J_2$, we make the familiar estimate through Parseval's theorem,
the Cauchy-Schwarz inequality, and then \req{Rtilsig} to obtain
	\begin{align}\notag
		|\J_2|&\leq \kap
                F_{H^{\s-\gam/2}}^2+\frac{{{{\kap}}}}{\eric 10}\Sob{\eta}{H^{\s+\gam/2}}^2.
	\end{align}
For $\J_3$ and $\J_4$, we consider two cases: $\s\leq\gam/2$ and $\s>\gam/2$.

{\flushleft\textit{Case: $\s\leq\gam/2$:}}
It follows from Fubini's theorem, H\"older's inequality, \req{Ih:base0}, and the Poincar\`e inequality that 
	\begin{align}\notag
		\left|\int
                J_h^{{\de}}\eta\Lam^{2\s}\eta\ dx\right|&\leq
                \frac{1}\de\int_{t-2\de}^{t-\de}\Sob{J_h
                  \eta(s)}{L^2}\Sob{\eta(t)}{H^{2\s}}\ ds\notag\\
                  	&\leq\left(\sup_{s\in I_{k-2}\cup I_{k-1}}\Sob{\eta(s)}{L^2}\right)\Sob{\eta(t)}{H^{\s+\gam/2}}\notag\\
			&\leq CM_{L^2}\Sob{\eta(t)}{H^{\s+\gam/2}}\notag
	\end{align}
Thus, by Young's inequality we have
	\begin{align}\notag
		 |\J_3|&\leq C\frac{{\mu}^2}{{\kap}}M_{L^2}^2
                 +\frac{{\kap}}{\eric 10}\Sob{\eta}{H^{\s+\gam/2}}^2.
	\end{align}
Similarly, since $\tht_{-2{\de}}\in\mathcal{B}_{L^2}$ by $(H5)$, by
\req{sqg:hs:ball} we have
	\[
	|\J_4|\leq C\frac{{\mu}^2}{{\kap}}\Tht_{L^2}^2
	+\frac{{\kap}}{\eric 10}\Sob{\eta}{H^{\s+\gam/2}}^2.
	\]
	
Therefore, upon returning to \req{balance:hs}, then applying the estimates
for {$\J_1$ through $\J_4$} and the Poincar\'e inequality gives
	\begin{align}
		\frac{d}{dt}\Sob{\eta}{H^\s}^2+{{{{\kap}}}}\Sob{\eta}{H^{\s}}^2\leq
                {\eric 8}{\kap}F_{{H^{\s-
                    \gam/2}}}^2+C\kap\til{\Xi}_{L^p}^{2+\gam/\s}M_{L^2}^2+{C}\frac{{{{\mu}}}^2
                }{{{{\kap}}}}\left(M_{L^2}^2+\Tht_{L^2}^2\right).\notag
	\end{align}
Then the Gronwall inequality implies
	\begin{align}
		\Sob{\eta(t)}{H^\s}^2&+\int_0^te^{-\kap(t-s)}\Sob{\eta(s)}{H^{\s+\gam/2}}^2\ ds\notag\\
			&\leq \Gam_{H^\s}^2e^{-\kap t}+C\left[\left(\til{\Xi}_{L^p}^{2+\gam/\s}+\frac{1}{h^{2\gam}}\right)M_{L^2}^2+\frac{1}{h^{2\gam}}R_{H^\s}^2\right]\left(1-e^{-\kap t}\right),\notag
	\end{align}
as desired.

{\flushleft{\textit{Case: $\s>\gam/2$:}}}
Observe that by Fubini's theorem,
Plancherel's theorem, H\"older's inequality,  {\eric \req{TI:3}},
Proposition \ref{prop:l2:avg},
 and Young's inequality we have
	\begin{align}
		|\J_3|&\leq \mu\Sob{J_h^\de\eta(t)}{H^{\s-\gam/2}}\Sob{\eta(t)}{H^{\s+\gam/2}}\notag\\
		&\leq\mu\left(\frac{1}{\de}\int_{t-2\de}^{t-\de}\Sob{J_h\eta(s)}{H^{\s-\gam/2}}\ ds\right)\Sob{\eta(t)}{H^{\s+\gam/2}}\notag\\
		&\leq C\mu h^{-(\s-\gam/2)}\left(\sup_{s\in I_{k-2}\cup I_{k-1}}\Sob{\eta(s)}{L^2}\right)\Sob{\eta(t)}{H^{\s+\gam/2}}\notag\\
		&\leq C\frac{\mu}{h^{2\s}}\frac{\mu h^{\gam}}{\kap}M_{L^2}^2+\frac{\kap}{\eric 10}\Sob{\eta(t)}{H^{\s+\gam/2}}^2.\notag
	\end{align}
Similarly, since $\tht_{-2{\de}}\in\mathcal{B}_{L^2}$ by $(H5)$, by
\req{sqg:hs:ball} we have
	\[
	|\J_4|\leq C\frac{\mu}{h^{2\s}}\frac{\mu h^{\gam}}{\kap}\Tht_{L^2}^2+\frac{\kap}{\eric 10}\Sob{\eta(t)}{H^{\s+\gam/2}}^2.
	\]

Therefore, upon returning to \req{balance:hs}, then applying the estimates
for $\J_1$ through $\J_4$ and the Poincar\'e inequality gives
	\begin{align}
		\frac{d}{dt}\Sob{\eta}{H^\s}^2+{{{{\kap}}}}\Sob{\eta}{H^{\s}}^2\leq
                {\eric 8}{\kap}F_{H^{\s+
                    \gam/2}}^2+C\kap\til{\Xi}_{L^p}^{2+\gam/\s}M_{L^2}^2+C\frac{\mu^2}{h^{2\s-\gam}\kap}\left(M_{L^2}^2+\Tht_{L^2}^2\right),\notag
	\end{align}
Then the Gronwall inequality and \req{mu:hgam:hs} implies
	\begin{align}
		\Sob{\eta(t)}{H^\s}^2&+\int_0^te^{-\kap(t-s)}\Sob{\eta(s)}{H^{\s+\gam/2}}^2\ ds\notag\\
			&\leq \Gam_{H^\s}^2e^{-\kap t}+C\left[F_{H^{\s-\gam/2}}^2+\left(\til{\Xi}_{L^p}^{\frac{2\s+\gam}\s}+\frac{1}{h^{2\s+\gam}}\right)M_{L^2}^2+\frac{1}{h^{2\s+\gam}}\Tht_{L^2}^2\right]\left(1-e^{-\kap t}\right),\notag
	\end{align}
as desired.
\end{proof}

\section{Proof of Theorem \ref{thm2}}\label{sect:pf2}

We are left to establish the synchronization of $\eta$ to the
reference solution $\tht$.  We point out that the uniform $L^2$ bounds will be used in a {crucial} way to establish
suitable control on the time derivative and guarantee synchronization in a rather weak topology, i.e., the $H^{-1/2}$ topology.  We then make use of the uniform $L^p$ and $H^\s$-bounds in order to
strengthen the regularity of the convergence of the synchronization by
interpolation.

Consider the difference $\z:=\eta-\tht$, where $\tht\in
\mathcal{B}_{H^\s}$ and $\eta$ is the unique strong solution of
\req{fb:sqg:avg}.  Observe that \req{gam:bounds} ensures that $\z$ is
defined for $t \in I_{-1}$.  The evolution of $\z$ is given by:
	\begin{align}\label{z:eqn2}
		\begin{cases}
			&\bdy_t\z+{{{\kap}}}\Lam^\gam\z+w\cdotp{\del}\z+w\cdotp{\del}
                  \tht+u\cdotp{\del}\z=-{{{\mu}}}
                  J_h^{{\de}}\z,\\
                  &w=\Ri^\perp\z,\quad \z(t)=g(t)-\tht(t),\quad
                  t\in(-2{\de},0].
		\end{cases}
	\end{align}
It will be convenient to work at the regularity level of the stream function of $\z$.  Thus, we define
	\begin{align}\label{stream:def}
		\psi:=-\Lam^{-1}\z.
	\end{align}


\subsection{Synchronization}
Our main claim is the following.

\begin{prop}\label{prop:synch}
Let $\Tht_{H^\s}, \Tht_{L^2}, \Tht_{L^p}$ {\eric and} $M_{L^2}$ be given
by \req{sqg:hs:ball}, \req{tht:bounds} and \req{M2:def:final}.  Define
	\begin{align}\label{xi:tht}
		\Xi_{L^p}:=\left(\frac{{\Tht}_{L^p}}{{\kap}}\right)^{\frac{\gam/2}{\gam-1-2/p}},\qquad {\Psi}:=4\sqrt{2}M_{L^2},
	\end{align}
	\begin{align}\label{c1c2a}
		\til{C}_1(h)&:=\kap^3\left(\frac{2\pi}h\right)^2\left(\frac{1}{h^{1+3\gam}}+\frac{1}{\kap^2}\frac{1}{h^{4+\gam}}\right)(1+M_{L^2}^2+\Tht_{L^2}^2)
	\end{align}
and
	\begin{align}\label{c1c2b}
\til{C}_2(h)&:=\left(\frac{2\pi}{h}\right)^2\frac{1}{h^{4+\gam}}\left(M_{L^2}^2+\Tht_{L^2}^2\right).
	\end{align}
There exist constants {$c_0, c_0',c_1,c_2\ge 1$} such that if $h, \mu$ satisfies
	\begin{align}\label{sync:cond}
		\frac{1}{c_0'}\Xi_{L^p}^2\leq\frac{\mu}{\kap}\leq \frac{1}{c_0}h^{-\gam}.
	\end{align}
and $\de$ is chosen to satisfy
	\begin{align}\label{de:sync:avg}
		\frac{1}{\kap}\frac{\de^2\til{C}_1(h)+\de\til{C}_2(h)}{\Xi_{L^p}^2}\leq\frac{c_0'}{c_1}\quad\text{and}\quad \de\leq\frac{1}{c_2^{1/2}}\min\left\{c_1^{1/2},\frac{h^\gam}{\kap},\frac{\kap^{1/2}}{\til{C}_2(h)^{1/2}}\right\}.
	\end{align}
then
	\begin{align}\label{sync:claim}
		\Sob{\psi(t)}{H^{1/2}}^2\leq \Psi^2
                e^{-({{{\mu}}}/4)(t-{\de})},\quad t\geq 2{\de}.
	\end{align}
\end{prop}
	
To prove this, we proceed as in section \ref{sect:l2:avg} and make some preparatory estimates.

\subsubsection{Control of temporal oscillations at a fixed spatial scale}

\begin{lem}\label{lem:dt3}
Let $\Tht_{L^2}$ and $M_{L^2}$ be given by \req{tht:bounds}, \req{M2:def}, respectively.
Let $c_0>0$ be a constant.  Suppose that
	\begin{align}\label{mu:cond:synch:dt}
		\frac{\mu h^\gam}{\kap}\leq \frac{1}{c_0},
	\end{align}
Then there exists a constant $C_0>0$, depending on $c_0$, such that
	\begin{align}\label{dt:synch}
	\Sob{\bdy_tJ_h\z(t)}{H^{-\gam/2}}^2
	&\leq C_0\left(\frac{2\pi}h\right)^2\frac{\kap^2}{h^{1+\gam}}\left(\frac{1}{\de}\int_{t-2\de}^{t-\de}\Sob{\psi(s)}{H^{1/2}}^2\ ds\right)\notag\\
	&+ C_0\kap^2\left(\frac{2\pi}h\right)^2\left(M_{L^2}^2
		+\Tht_{L^2}^2\right)\left(\frac{1}{h^{\gam+1}}
		+\frac{1}{\kap^2}\frac{1}{h^{4-\gam}}\right)\Sob{\psi}{H^{1/2}}^2
		\\
	&+C_0\left(\frac{2\pi}h\right)^2\left(M_{L^2}^2+\Tht_{L^2}^2\right)h^{-(4-\gam)}\Sob{\psi}{H^{(\gam+1)/2}}^2,
\quad \hbox{for}\quad t>-2\de.\notag
	\end{align}
\end{lem}

\begin{proof}
Let $t>-2\de$.  {
Applying $J_h$ to \eqref{z:eqn2} and taking the $H^{-\gamma/2}$-norm yields
\begin{align*}
	\|\partial_t J_h\zeta\|_{H^{-\gamma/2}}
	&\le
		\kappa \|J_h\Gamma^{\gamma}\zeta\|_{H^{-\gamma/2}}
		+\mu\|J_h J_h^{\delta}\zeta\|_{H^{-\gamma/2}}\\
		&+\|J_h\nabla\cdot (w\zeta)\|_{H^{-\gamma/2}}
		+\|J_h\nabla\cdot (w\theta)\|_{H^{-\gamma/2}}
		+\|J_h\nabla\cdot (u\zeta)\|_{H^{-\gamma/2}}.
\end{align*}
}%
Observe that by $(H1)$, we have $\gam/2<1$, so that by \req{typeI:const},
we have
	\begin{align}\notag
		C_I(\gam/2,h)=C\left(\frac{2\pi}h\right).
	\end{align}

By \req{TI:5}, \req{tht:bounds}, \req{stream:def}, the Cauchy-Schwarz inequality, and \req{mu:cond:synch:dt} we have
	\begin{align}
		\kap\Sob{J_h\Lam^\gam\z(t)}{H^{-\gam/2}}&\leq C\kap\left(\frac{2\pi}h\right)h^{\gam/2-\gam-1/2}\Sob{\Lam^\gam\z}{H^{-\gam-1/2}}\notag\\
			&\leq C\kap\left(\frac{2\pi}h\right)h^{-(\gam+1)/2}\Sob{\psi}{H^{1/2}},\notag\\
		\mu\Sob{J_hJ_h^\de\z(t)}{H^{-\gam/2}}&\leq \frac{\mu}\de\int_{t-2\de}^{t-\de}\Sob{J_h\z(s)}{H^{-\gam/2}}\ ds\notag\\
			&\leq C\left(\frac{2\pi}h\right){\mu}h^{(\gam-1)/2}\left(\frac{1}{\de}\int_{t-2\de}^{t-\de}\Sob{\z(s)}{H^{-1/2}}\ ds\right)\notag\\
			 &\leq C\left(\frac{2\pi}h\right)\left(\frac{\mu}{\de^{1/2}}\right)h^{(\gam-1)/2}\left(\int_{t-2\de}^{t-\de}\Sob{\psi(s)}{H^{1/2}}^2\ ds\right)^{1/2}\notag\\
			 &\leq C\left(\frac{2\pi}h\right)\frac{\kap}{h^{(1+\gam)/2}}\left(\frac{1}{\de}\int_{t-2\de}^{t-\de}\Sob{\psi(s)}{H^{1/2}}^2\ ds\right)^{1/2}\notag.
	\end{align}

To estimate the nonlinear terms, we apply \req{TI:9}, the Cauchy-Schwarz inequality, \req{tht:bounds}, Proposition \ref{prop:l2:avg}, \req{stream:def}, interpolation, and Young's inequality.  For instance, we have
	\begin{align}
	\Sob{J_h\del\cdotp(w\z)}{H^{-\gam/2}}&\leq C\left(\frac{2\pi}{h}\right)h^{-2+\gam/2}\Sob{(\Ri^\perp\z)\z}{L^1}\notag\\
				&\leq C\left(\frac{2\pi}{h}\right)h^{-2+\gam/2}\Sob{\z}{L^2}^2\notag\\
			&\leq C\left(\frac{2\pi}h\right)h^{-2+\gam/2}(M_{L^2}+\Tht_{L^2})\left(\Sob{\psi}{H^{(\gam+1)/2}}+\Sob{\psi}{H^{1/2}}\right).\notag
	\end{align}
Similarly
	\begin{align}
	{\Sob{J_h\del\cdotp(w\tht)}{H^{-\gam/2}}},
	{\Sob{J_h\del\cdotp(u\z)}{H^{-\gam/2}}} \leq C\left(\frac{2\pi}{h}\right)h^{-2+\gam/2}\Tht_{L^2}\left(\Sob{\psi}{H^{(\gam+1)/2}}+\Sob{\psi}{H^{1/2}}\right)\notag.
	\end{align}
Therefore,
by summing each of these estimates, we arrive at \req{dt:synch}.
as desired.
\end{proof}

\subsubsection{Growth during transient period}
We introduce the following notation: Let $\al\in(0,1)$ and
$\ell\in\Z$, then define
	\begin{align}\notag
		\de_{\al \ell}:=\al \ell\de.
	\end{align}
Observe that by the Poincar\`e inequality, \req{M:til1} implies
	\begin{align}\notag
		\Sob{\psi(t)}{H^{1/2}}^2+{{{\kap}}}\int_{I_k}^te^{-({{{\mu}}}/2)(t-s)}\Sob{\psi(s)}{H^{1/2}}^2\ ds
                \leq M_{L^2}^2e^{(\mu/2)\de}\leq 32M_{L^2}^2,\quad t\in I_k, \quad k\geq-1.
	\end{align}
Clearly, one has
	\begin{align}\notag
 M_{L^2}^2e^{(\mu/2)\de}\leq M_{L^2}^2e^{{{{(5\mu/2)}}}{\de}}e^{-({{{\mu}}}/2)(t-{{\de}_{k/2}})}\leq 32M_{L^2}^2e^{-(\mu/2)(t-{\de_{k/2}})},\quad
 t\in I_k, \quad k=-1, 0, 1.
	\end{align}
Then
	\begin{align}\label{bc:sync}
		\Sob{\psi(t)}{H^{1/2}}^2+\frac{{{\kap}}}2\int_{{\de}_{k}}^te^{-({{{\mu}}}/2)(t-s)}\Sob{\psi(s)}
                    {H^{(\gam+1)/2}}^2\ ds\leq{\Psi}^2e^{-({{{\mu}}}/2)(t-{\de}_{k/2})},\quad
                    t\in I_{k},\quad k=-1, 0, 1.
	\end{align}
We are now ready to prove the synchronization property.

\subsection{Proof of Proposition \ref{prop:synch}}

\subsubsection*{Proof of Proposition \ref{prop:synch}}

We proceed by induction on $k$ with the base case, $k=1$, as
established by \req{bc:sync}.
{Suppose} that the following holds:
		\begin{align}\label{ind:hyp:sync}
		\Sob{\psi(t)}{H^{1/2}}^2+\frac{{{\kap}}}{2}\int_{{\de}_{k}}^te^{-({{{\mu}}}/2)(t-s)}\Sob{\psi(s)}
                    {H^{(\gam+1)/2}}^2\ ds\leq{\Psi}^2e^{-({{{\mu}}}/2)(t-{\de}_{k/2})}
\end{align}
{for $t\in I_{\ell}$ and $\ell=0,\dots,k,$}
where $\Psi$ is given by \req{xi:tht}.  We show that this
corresponding bound holds over $I_{k+1}$ as well.

Let $t\in I_{k+1}$, $k\geq1$.  {Multiply} \req{z:eqn2} by $\psi$
and integrate over $\T^2$ to obtain
	\begin{align}\label{balance:sync}
		\frac{1}2\frac{d}{dt}\Sob{\psi}{{H}^{1/2}}^2&+{{{\kap}}}\Sob{\psi}{H^{(\gam+1)/2}}^2+{{{\mu}}}
                \Sob{\psi}{{H}^{1/2}}^2\notag\\ &={\int
                 (u\cdotp\del\psi)\z dx}+{{{{\mu}}} \int
                  (\z-J_h\z)\psi dx}+{\mu}
                \int(J_h\z-J_h^{{\de}}\z)\psi\ dx\notag\\
		&=\K_1+\K_2+\K_3.
	\end{align}
Note that we have used the orthogonality property, i.e., $\Ri^\perp f\cdotp\Ri f=0$.

We refer to \cite{const:wu:qgwksol, resnick} to estimate $\K_1$.  In
particular, by H\"older's inequality, the Calder\`on-Zygmund theorem,
and Sobolev embedding, $H^{1/p}\imb L^q$, we have
	\begin{align}\label{z:nlt:1}
		\left|\K_1\right|\leq
                C\Sob{u}{L^p}\Sob{\z}{L^{q}}\Sob{{\del}\psi}{L^q}\leq
                C\Sob{\tht}{L^p} \Sob{\psi}{H^{1+1/p}}^2,
	\end{align}
where $1/p+2/q=1$.  Since $p>2/(\gam-1)$ by $(H3)$, by interpolation
we have
	\begin{align}\notag
		\Sob{\psi}{H^{1+1/p}}\leq
                C\Sob{\psi}{H^{(\gam+1)/2}}^{\frac{1+2/p}{\gam}}\Sob{\psi}{{H}
                  ^{1/2}}^{\frac{\gam-1-2/p}{\gam}}.
	\end{align}
Thus, by Young's inequality we obtain
	\begin{align}\notag
		\left| \K_1\right|\leq
                \frac{{{{\kap}}}}{6}\Sob{\psi}{H^{(\gam+1)/2}}^2+C{{\kap}}\Xi_{L^p}^2\Sob{\psi}
                     {H^{1/2}}^2.
	\end{align}
where $\Xi_{L^p}$ is given by \req{xi:tht}.  We estimate $\K_2$ with the Parseval's theorem, the Cauchy-Schwarz inequality, \req{TI}, \req{stream:def}, interpolation, and Young's inequality to get
	\begin{align}\notag
		|\K_2|&\leq \mu\Sob{\z-J_h\z}{H^{-\gam/2}}\Sob{\psi}{H^{\gam/2}}\notag\\
		&\leq \mu h^{\gam/2}\Sob{\psi}{H^1}\Sob{\psi}{H^{\gam/2}}\notag\\
		&\leq\mu h^{\gam/2}\Sob{\psi}{H^{(\gam+1)/2}}\Sob{\psi}{H^{1/2}}\notag\\
               &\leq \frac{{{{\kap}}}}{6}\Sob{\psi}{H^{(\gam+1)/2}}^2+C\frac{{{{\mu}}}
                  ^2h^\gam}{{{{\kap}}}}\Sob{\psi}{{H}^{1/2}}^2.\notag
	\end{align}
For $\K_3$, similar to \req{mvt:bc}, we estimate
	\begin{align}\notag
		|\K_3|&\leq{C{\de}
                 \frac{{{{\mu^2}}}}{\kap}
                  \int_{t-2{\de}}^t\Sob{\bdy_sJ_h\z(s)}{H^{-\gam/2}}^2\ ds}
                  +\frac{\kap}4\Sob{\psi}{H^{\gam/2}}^2\notag\\
			&\leq {C{\de}
                 \frac{{{{\mu^2}}}}{\kap}
                  \int_{t-2{\de}}^t\Sob{\bdy_sJ_h\z(s)}{H^{-\gam/2}}^2\ ds}
                  +\frac{{{{\kap}}}}6\Sob{\psi}{H^{(\gam+1)/2}}^{2}.\notag
	\end{align}
Returning to \req{balance:sync} and combining {\eric $\K_1$ through $\K_3$},
	then applying
\req{sync:cond} with $c_0$ {\eric and} $c_0'$ sufficiently large, we get
	\begin{align}\label{balance:star}
		\frac{d}{dt}\Sob{\psi}{{H}^{1/2}}^2&+{{{\kap}}}\Sob{\psi}{H^{(\gam+1)/2}}^2+{{{\mu}}}
                \Sob{\psi}{{H}^{1/2}}^2\leq \mystarb(t),
	\end{align}
where
	\begin{align}\notag
		\mystarb(t):={C{\de}\frac{\kap}{h^{2\gam}}\int_{t-2{\de}}^t\Sob{J_h\bdy_s\psi(s)}{H^{-\gam/2}}^2\ ds}.
	\end{align}

Observe that $\mystarb(t)\leq\mystarb_{k-1}+\mystarb_{k}+\mystarb_{k+1}(t)$,
where
	\begin{align}\notag
		\mystarb_\ell(t):={C{{\de}}\frac{\kap}{h^{2\gam}}\int_{{\de}_\ell}^{t}
            \Sob{J_h\bdy_s\psi(s)}{H^{-\gam/2}}^2\ ds}
	\qquad\hbox{and}\qquad
		\mystarb_\ell:=\mystarb_\ell(\delta_{\ell+1}).
	\end{align}
%
Let $\ell\in\{k-3, k-2, k-1,k\}$.  By the second condition in \req{de:sync:avg}, with $c_2$ chosen large
enough, we have ${\de}{\mu} \leq C^{-1}(\ln 4)$, so that Lemma
\ref{lem:prev:sync} guarantees
  that
	\begin{align}\label{psi:avg1}
		\frac{1}{\de}\int_{I_{\ell}}\Sob{\psi(s)}{H^{1/2}}^2\ ds\leq
                 C\Psi^2e^{-(\mu/2)(t-{\de}_{\ell'/2})},\quad \ell'\in(\ell,\ell+N],\quad N=3,
	\end{align}
as well as
	\begin{align}\label{psi:avg2}
		{\kap}\int_{I_{\ell}}\Sob{\psi(s)}{H^{(\gam+1)/2}}^2\ ds\leq
                C\Psi^2e^{-({\mu}/2)(t-{\de}_{\ell'/ 2})},\quad \ell'\in(\ell,\ell+N],\quad N=3.
	\end{align}
Thus, by Lemma \ref{lem:dt3} and
\req{ind:hyp:sync}, \req{psi:avg1}, and \req{psi:avg2} we have

	\begin{align}\label{kpm1}
		\mystarb(t)\leq&
    			C\de^2\til{C}_1(h)\sum_{\ell=k-3}^k\frac{1}{\de}\int_{I_{\ell}}\Sob{\psi(s)}{H^{1/2}}^2\ ds
	+C\de\til{C}_2(h)
		\sum_{\ell=k-1}^k\kap
		\int_{I_{\ell}}\Sob{\psi(s)}{H^{(\gam+1)/2}}^2\ ds
	\notag\\&
	+C\de\til{C}_1(h)\int_{\de_{k+1}}^t\Sob{\psi(s)}{H^{1/2}}^2\ ds
	+C\de\til{C}_2(h)\kap\int_{\de_{k+1}}^t\Sob{\psi(s)}{H^{(\gam+1)/2}}^2\ ds
	\notag\\
			\leq &\til{O}(\de)\Psi^2e^{-(\mu/2)(t-\de_{k/2})}
	\notag\\&
	+\til{O}_1(\de)\int_{\de_{k+1}}^t\Sob{\psi(s)}{H^{1/2}}^2\ ds
	+\til{O}_2(\de)\kap\int_{\de_{k+1}}^t\Sob{\psi(s)}{H^{(\gam+1)/2}}^2\ ds.
	\end{align}
where $\til{C}_1(h), \til{C}_2(h)$ are given by \req{c1c2a}, \req{c1c2b} and
	\begin{align}\label{more:quants}
		\begin{split}
			\til{O}({\de})&:=C(\til{O}_1(\de^2)+\til{O}_2(\de)),\quad \til{O}_1(\de):=\de\til{C}_1(h),\quad \til{O}_2(\de):=C\de\til{C}_2(h).
		\end{split}
	\end{align}
for some constant $C>0$.


Returning to \req{balance:star} and combining \req{kpm1}
gives
	\begin{align}
		\frac{d}{dt}\Sob{\psi}{H^{1/2}}^2&+{{{\kap}}}\Sob{\psi}{H^{(\gam+1)/2}}^2+{{{\mu}}}\Sob{\psi}
                     {H^{1/2}}^2\notag\\
                     &\leq \til{O}(\de)\Psi^2e^{-(\mu/2)(t-\de_{k/2})}
\notag\\&
\quad +\til{O}_1(\de)\int_{\de_{k+1}}^t\Sob{\psi(s)}{H^{1/2}}^2\ ds+\til{O}_2(\de)\left(\kap\int_{\de_{k+1}}^t\Sob{\psi(s)}{H^{(\gam+1)/2}}^2\ ds\right).\notag
	\end{align}
Hence, provided that $c_1, c_2$ are chosen sufficiently large with $c_2$ depending on $c_1$, it follows from  \req{de:sync:avg} that Lemma \ref{gron} (i) applies
over $t\in I_{k+1}$ with
	\[
		a=\mu,\quad b=\kap,\quad
                A=C(\de\til{C}_1(h)+\til{C}_2(h)),\quad
                B=C\til{C}_2(h),\quad
                F=\til{O}(\de)\Psi^2e^{-(\mu/2)(t-\de_{k/2})}.
	\]
In particular, Lemma \ref{gron} (i) implies
	\begin{align}\notag
		\Sob{\psi(t)}{H^{1/2}}^2&+\frac{{{\kap}}}2\int_{{\de}_{k+1}}^te^{-({{{\mu}}}/2)(t-s)}\Sob{\psi(s)}
                    {H^{(\gam+1)/2}}\ ds\notag\\ &\leq
                    \Sob{\psi({\de}_{k+1})}{H^{1/2}}^2e^{-({{{\mu}}}/2)(t-{\de}_{k+1})}+\til{O}({\de})
                    \Psi
                    e^{-({{{\mu}}}/2)(t-{\de}_{k/2})}(t-{\de}_{k+1}).\notag
	\end{align}
By \req{ind:hyp:sync}, we have
	\begin{align}\notag
		\Sob{\psi({\de}_{k+1})}{H^{1/2}}^2e^{-({{{\mu}}}/2)(t-\de_{k+1})}\leq
                \Psi^2e^{-({{{\mu}}}/2)
                  (\de_{k+1}-\de_{k/2})}e^{-({{{\mu}}}/2)(t-\de_{k+1})}=\Psi^2e^{-({{{\mu}}}/2)(t-\de_{k/2})},\quad
                t\in I_{k +1}.
	\end{align}
Also, we have
	\begin{align}\notag
		\til{O}({\de})\Psi
                e^{-({{{\mu}}}/2)(t-{\de}_{k/2})}(t-{\de}_{k+1})\leq
                \de \til{O}({\de})\Psi e^{-
                  ({{{\mu}}}/2)(t-{\de}_{k/2})}.
	\end{align}
Since
	\[
		e^{-(\mu/2)(t-\de_{k/2})}=e^{-({{{\mu}}}/4)\de}e^{-({{{\mu}}}/2)(t-{\de}_{(k+1)/2})}
	\]
It follows that
	\begin{align}\notag
		\Sob{\psi(t)}{H^{1/2}}^2
		&+\frac{{{\kap}}}2\int_{{\de}_{k+1}}e^{-({{{\mu}}}/2)(t-s)}\Sob{\psi(s)}
                    {H^{(\gam+1)/2}}^2\ ds\\
	&\leq \Psi^2\left(1+\de
                    \til{O}({\de})\right)e^{-({{{\mu}}}/4)\de}e^{-({{{\mu}}}/2)(t-
                      {\de}_{(k+1)/2})},\quad t\in I_{k+1}.
	\end{align}
Observe that \req{de:sync:avg} with $c_1$ chosen sufficiently large
ensures $1+\de \til{O}({\de})\leq e^{({{{\mu}}}/4){\de}}$.  This
establishes \req{ind:hyp:sync} for $k+1$.  Through Lemma \ref{gron} (ii), we may iterate this bound to deduce \req{sync:claim}, as desired.  \qed

\subsection{Proof of Theorem \ref{thm2}}
Under the Standing Hypotheses, Theorem \ref{thm1} guarantees a unique,
global strong solution $\eta $ of \req{fb:sqg:avg}.  Let $c_0$ denote
the maximum among all the constants, $c_0, c_0'$, appearing in Propositions
\ref{prop:l2:avg} and \ref{prop:synch}.
Then let $c_1, c_2$ denote the maximum among all the $c_1, c_2$ appearing in those propositions as well (possibly choosing $c_2$ larger). Suppose that $\mu, h$ satisfies
	\begin{align}\label{mu:xi:cond}
		\frac{1}{c_0'}\Xi_{L^p}^2\leq\frac{\mu}{\kap}\leq \frac{1}{c_0}h^{-\gam}.
	\end{align}
Choose $\de$ so that \req{de:small:l2},
\req{de:sync:avg} are satisfied, and is chosen smaller than
	\begin{align}\notag
			\frac{1}{c_1}\left(\frac{h}{2\pi}\right)\min\left\{\frac{h^{\gam/2}}{(c_0')^{1/2}}\Xi_{L^p}\frac{h^2}{M_{L^2}},\frac{h^\gam}{\kap}\right\}.
	\end{align}
Then \req{mu:xi:cond} implies that \req{gron:cond:l2} holds as well.  Thus, upon applying Propositions \ref{prop:l2:avg} and
\ref{prop:synch}, $\eta$ satisfies
	\begin{align}\notag
		\Sob{\eta(t)-\tht(t)}{H^{-1/2}}\leq
                O(e^{-(\mu/4)(t-2\de)}),\quad t>2\de.
	\end{align}
Observe that Propositions \ref{prop:ga}, \ref{prop:lp:avg}, and \ref{prop:bad:hs} then imply
that
	\[
	\sup_{t>-2\de}\Sob{\eta(t)-\tht(t)}{H^\s}\leq\til{M}_{H^\s}(h)+\Tht_{H^\s},
	\]
where
	\begin{align}
		\til{M}_{H^\s}(h):=\Gam_{H^\s}^2+C_0\left[\left(\til{\Xi}_{L^p}^{\frac{2\s+\gam}\s}+\frac{1}{h^{2\s+\gam}}\right)M_{L^2}^2+\frac{1}{h^{2\s+\gam}}+\frac{1}{h^{2\s+\gam}}R_{H^\s}^2\right]\notag,
	\end{align}
for some sufficiently large constant $C_0>0$.  Therefore, for each $\s'<\s$, by
interpolation, there exists a constant $\lam_0=
\lam_0(\s')\in(0,1)$ such that
	\begin{align}\notag
		\Sob{\eta(t)-\tht(t)}{H^{\s'}}\leq
                {O(e^{-\lam_0\mu(t-2\de)})},\quad t>2\de.
	\end{align}
Choosing $\s'=0$, yields the desired convergence in $L^2$.

\qed

\subsection{Concluding remarks}\label{methodology}
Depending on the type of measurement, the size of the
averaging window that effectively blurs the observations in
time may be quite different.
For example, radiometers and hot-wire anemometers may produce
data with averages in the microsecond range.
Velocities obtained from mechanical
weather-vane anemometers may be averaged with respect to
a time window measured in seconds, while
velocity data obtained from the Lagrangian trajectories
of buoys placed in the ocean is likely to include time averages
measured in hours if not days.
Observations of temperatures are similar.
As we saw, it is important for our analysis that the size of the
time-averaging window is not too large.
Intuitively speaking, the length of the averaging window
should be smaller than any dynamically relevant timescales
in the flow.
Numerical computations involving the Lorenz
system \cite{blocher} show that synchronization occurs when
the averaging window is of size $\delta=0.25$ which, poetically
speaking, is
about ten times smaller than the time it takes to travel
around one wing of the butterfly.
In the case of the fluids, we conjecture
that the averaging window should be at least ten times
smaller than the turnover time of the smallest physically
relevant eddy.
Alternatively, the largest averaging window such that
our data assimilation algorithm leads to full recovery of
the observed solution could be interpreted as a definition
of the smallest physically relevant time scale.


We reiterate that a main motivation to consider a more realistic representation of
physical observations is the reason for considering time averages.
The additional $\delta$ delay introduced into equations (\ref{fb:sqg:prelim})
helps close the estimates in the analysis while being of the same
magnitude as the $\delta/2$ delay dictated by causality considerations
in the feedback controller (see Remark \ref{rmk:delay}).
In practice, such a delay may also be used to advance an
initial condition already obtained by data assimilation for
a short time into the future to increase the stability of
further predictions.
However, this idea must be left for a different study.

\subsection*{\bf Acknowledgments}  The authors would like to thank the Institute of Pure and Applied Mathematics (IPAM) at UCLA for the warm hospitality where this collaboration was conceived. The authors are also thankful to Thomas Bewley, Aseel Farhat, and Hakima Bessaih for the insightful discussions.
M.S.\ Jolly was supported by NSF grant DMS-1418911.
 E.\ Olson was supported by NFS grant DMS-1418928.
 The work of E.S.\ Titi was supported
in part by ONR grant N00014-15-1-2333, the Einstein
Stiftung/Foundation--Berlin, through the Einstein Visiting
Fellow Program, and by the John Simon Guggenheim Memorial Foundation.

\appendix
\section{}\label{appA}

To obtain the uniform estimates, we invoked a non-local Gronwall
inequality, which ensured such bounds provided that the non-local term
was sufficiently small.

\begin{lem}\label{gron}
Let ${\Ph},{{\Psi}}, F$ be non-negative, locally integrable functions
on $(t_0,t_0+\de]$ for some $t_0\in \R$ and $\de>0$ such that
	\begin{align}\label{d:ineq}
		\frac{d}{dt}{\Ph}+a{\Ph}+b{{\Psi}}\leq F+A\de
                \int_{t_0}^t{\Ph}(s)\ ds+B\de\int_{t_0}^t{{\Psi}}
                (s)\ ds,\quad t\in(t_0,t_0+\de),
	\end{align}
for some $a,b,A, B>0$.  Suppose that $\de, a, c$ satisfy
	\begin{align}\label{de:cond:gron}
		\de\left(e^{(a/2)\de}-1\right)\leq\frac{a}{4}\min\left\{\frac{a}{A},\frac{b}{B}\right\},
	\end{align}
where we use the convention that $a/A=\infty$, $b/B=\infty$ if $A=0$, $B=0$, respectively.  Then the
following hold:

\begin{enumerate}[(i)]
	\item For all $t\in (t_0,t_0+\de]$:
	\begin{align}\label{unif:gron0}
		{\Ph}(t)+\frac{b}2\int_{t_0}^te^{-(a/2)(t-s)}{{\Psi}}(s)\ ds\leq
                e^{-(a/2)(t-t_0)}{\Ph}(t_0)+
                \int_{t_0}^te^{-(a/2)(t-s)}F(s)\ ds.
	\end{align}

	\item If ${\Ph}$ satisfies
	\begin{align}\label{unif:gron1}
		{\Ph}(t)\leq
                e^{-(a/2)(t-\de_0)}{\Ph}(\de_0)+\int_{\de_0}^te^{-(a/2)(t-s)}F(s)\ ds,\quad
                t \in(\de_0,t_0],
	\end{align}
	for some $\de_0<t_0$, then \req{unif:gron1} persists over
        $t\in(t_0,t_0+\de]$.

\end{enumerate}

\end{lem}

\begin{proof}
Multiplying by the factor $e^{(a/2){{{\kap}}} t}$, then integrating
over $[t_0,t]$, we obtain
	\begin{align}\notag
		{\Ph}(t)+\frac{a}2\int_{t_0}^t&e^{-(a/2)(t-s)}{\Ph}(s)\ ds+b\int_{t_0}^te^{-(a/2)(t-s)}{{\Psi}}(s)\ ds\notag\\ \leq&
                e^{-(a/2)(t-t_0)}{\Ph}(t_0)+\int_{t_0}^te^{-(a/2)(t-s)}F(s)\ ds\notag\\ &+A\de\int_{t_0}^te^{-(a/2){{{\kap}}}(t-s)}\int_{t_0}^s{\Ph}(\tau)\ d\tau\ ds+B\de\int_{t_0}
                ^te^{-(a/2){{{\kap}}}(t-s)}\int_{t_0}^s{{\Psi}}(\tau)\ d\tau\ ds,\notag
	\end{align}

Observe that
	\begin{align}\notag
		\frac{a}2\int_{t_0}^te^{-(a/2)(t-s)}{\Ph}(s)\ ds&\geq
                \frac{a}2e^{-(a/2)(t-t_0)}\int_{t_0}^t{\Ph}(s)
                \ ds\notag\\ A\de\int_{t_0}^te^{-(a/2)(t-s)}\int_{t_0}^s{\Ph}(\tau)\ d\tau\ ds&\leq
                \frac{2A\de}a\left(1-e^{-(a/
                  2)(t-t_0)}\right)\int_{t_0}^t{\Ph}(\tau)\ d\tau.\notag
	\end{align}
Similarly
	\begin{align}\notag
	\frac{b}2\int_{t_0}^te^{-(a/2)(t-s)}{{\Psi}}(s)\ ds&\geq
        \frac{b}2e^{-(a/2)(t-t_0)}\int_{t_0}^t{{\Psi}}(s)\ ds\notag\\ B\de\int_{t_0}^te^{-(a/2)(t-s)}\int_{t_0}^s{{\Psi}}(\tau)\ d\tau\ ds&\leq
        \frac{2B\de}b\left(1-e^{-
          (a/2)(t-t_0)}\right)\int_{t_0}^t{{\Psi}}(s)\ ds.\notag
	\end{align}

It follows that
	\begin{align}\notag
		\frac{a}2\int_{t_0}^te^{-(a/2)(t-s)}&{\Ph}(s)\ ds-c\de\int_{t_0}^te^{-(a/2)(t-s)}\int_{t_0}^s{\Ph}
                (\tau)\ d\tau\ ds\notag\\ &\geq
                \frac{a}2\left[1-\frac{4A\de}{a^2}\left(e^{(a/2)(t-t_0)}-1\right)\right]e^{-
                  (a/2)(t-t_0)}\int_{t_0}^t{\Ph}(s)\ ds\geq0,\notag
	\end{align}
provided that the first condition in \req{de:cond:gron} holds.  This
also holds with $b, B, \psi$, replacing $a, A, {\Ph}$, respectively,
provided the second condition in \req{de:cond:gron} holds.  This
implies \req{unif:gron0}.

Now assume that \req{d:ineq} holds over $(t_0,t_0+\de)$ and that
\req{unif:gron1} holds over $ [\de_0,t_0]$, for some $\de_0>0$.  Then
applying \req{unif:gron1} at $t_0$ to \req{unif:gron0} we have
	\begin{align}
		{\Ph}(t) &\leq
                {\Ph}(\de_0)e^{-(a/2)(t-\de_0)}+\int_{\de_0}^{t_0}e^{-(a/2)(t-s)}F(s)\ ds+\int_{t_0}
                ^te^{-(a/2)(t-s)}F(s)\ ds\notag,
	\end{align}
which simplifies to \req{unif:gron1}, as desired.
\end{proof}

We also made use of the following lemma in order to control feedback
effects that enter the present instant through a past time interval
and ultimately, ensure synchronization (see \req{balance:star}).
	
\begin{lem}\label{lem:prev:sync}
Let $\ell\geq-1$ and $N>0$.  Let $\de>0$ and define
$\de_\ell:=\ell\de$ and $I_\ell:=(\de_\ell, \de_{\ell+1}]$.
Let $\Ph, \Psi$ be non-negative, locally integrable functions.
 Suppose that for some $\ell\geq-1$, there exist constants $a, b,
 \Ph_0>0$, independent of $\ell, N$, such that
	\begin{align}\label{Ph:hyp}
		{\Ph}(t)+b\int_{\de_{\ell}}^te^{-({a}/2)(t-s)}{\Psi}(s)\ ds\leq
                {{\Ph}_0} e^{-({{{a}}}/2)(t-\de_{\ell/ 2})},\quad t\in
                I_{\ell}.
	\end{align}
If $\de$ satisfies
	\begin{align}\label{lem:de:sync}
		\de<\frac{c}{{{a}}},
	\end{align}
for some constant $c>0$, then there exists a constant $C_N>0$ such that
	\begin{align}\label{Ph:bound}
		\frac{1}{\de}\int_{I_\ell}{{\Ph}}(s)\ ds\leq
                C_N{{\Ph}_0} e^{-({{{a}}}/2)(t-\de_{\ell'/2})}\quad
                \ell'\in(\ell,\ell+N].
	\end{align}
and
	\begin{align}\label{Psi:bound}
		b\int_{I_\ell}{\Psi}(s)\ ds\leq
                {C}_N{\Ph}_0e^{-({{{a}}}/2)(t-\de_{\ell'/2})},\quad
                \ell'\in(\ell,\ell +N].
	\end{align}
\end{lem}

\begin{proof}
Observe that by the mean value theorem
	\begin{align}
		\int_{I_\ell}{{\Ph}}(s)\ ds&\leq{{\Ph}_0}\int_{\de_\ell}^{\de_{\ell+1}}e^{-({{{a}}}/2)(s-\de_{\ell/2})}
                \ ds\notag\\ &={{\Ph}_0}e^{({{{a}}}/2)\de_{\ell/2}}\frac{2}{{{a}}}
                \left(e^{-({{{a}}}/2)\de_{\ell}}-e^{-({{{a}}}/2)\de_{\ell+1}}\right)\notag\\ &={{\Ph}_0}e^{({{{a}}}/2)\de_{\ell/2}}e^{-({{{a}}}/
                  2)\de_{\ell+1}}\frac{2}{{{a}}}\left(e^{({{{a}}}/2)\de}-1\right)\notag\\ &={{\Ph_0}}e^{-({{{a}}}/2)\de_{\ell/2}}
                e^{-({{{a}}}/2)\de(1-\tht)}\de\notag.
	\end{align}
for some $0<\tht<1$, depending on $\de$.

By assumption on $\ell',\ell$, and the fact that $t\leq
\de_{\ell'+1}$, we have
	\begin{align}\label{exp:lk}
		e^{-({{{a}}}/2)\de_{\ell/2}}&=e^{-({{{a}}}/2)(t-\de_{\ell'/2})}e^{-({{{a}}}/2)\de_{\ell/2}}e^{({{{a}}}/2)
                  (t-\de_{\ell'/2})}\notag\\ &\leq
                e^{-({{{a}}}/2)(t-\de_{\ell'/2})}e^{({{{a}}}/2)(\de_{\ell'/2}-
                  \de_{\ell/2})} e^{({{{a}}}/2)\de}\notag\\ &\leq
                e^{({{{a}}}/2)\de(1+N/2)}e^{-({{{a}}}/2)(t-\de_{\ell'/2})}.
	\end{align}
Thus, by letting $C_N:=e^{(c/2)(1+ N/2)}$, \req{lem:de:sync} and
\req{exp:lk} imply \req{Ph:bound}.

On the other hand, observe that
	\begin{align}\notag
		b\int_{\de_\ell}^te^{-({{{a}}}/2)(t-s)}{\Psi}(s)\ ds\geq
                e^{-({{{a}}}/2)(t-\de_\ell)}b\int_{\de_{\ell}}
                ^t{\Psi}(s)\ ds.
	\end{align}
Upon application of \req{Ph:hyp}, we have
	\begin{align}\notag
		b\int_{\de_\ell}^t{\Psi}(s)\ ds\leq
                e^{({{{a}}}/2)(t-\de_\ell)}{\Ph}_0e^{-({{{a}}}/2)(t-\de_{\ell/2})}
                ={\Ph}_0e^{-({{{a}}}/2)\de_{\ell/2}},\quad t\in
                I_\ell.
	\end{align}
Thus, by \req{exp:lk} we have
	\begin{align}\notag
		b\int_{I_\ell}{\Psi}(s)\ ds\leq
                C_Ne^{-({{{a}}}/2)(t-\de_{\ell'/2})},
	\end{align}
and we are done.
\end{proof}

\section{}\label{appC}

\subsection{Partition of unity}
Let us briefly recall the partition of unity constructed in \cite{azouani:olson:titi} and used in \cite{jmt}.  To this end, we define for $\phi\in L^1(\T^2)$
	\begin{align}\label{glob:avg}
		\lb\phi\rb:=\frac{1}{4\pi^2}\int_{\T^2}\phi(x)\ dx.
	\end{align}

Let $N>0$ be a perfect square integer and partition $\Om$ into $4N$ squares of side-length $h=\pi/\sqrt{N}$.  Let $\mathcal{J}=\{0,\pm1,\pm2,\dots,\pm(\sqrt{N}-1), -\sqrt{N}\}^2$ and for each $\al\in\mathcal{J}$, define the semi-open square
	\begin{align}
		Q_\al=[ih,(i+1)h)\times[jh,(j+1)h),\quad\text{where}\quad \al=(i,j)\in\mathcal{J}.\notag
	\end{align}
Let $\mathcal{Q}$ denote the collection of all $Q_\al$, i.e.
	\begin{align}
		\mathcal{Q}:=\{Q_\al\}_{\al\in\mathcal{J}}.\notag
	\end{align}
	
Suppose that $N\geq9$ and $\eps=h/10$.  For each $\al=(i,j)\in\mathcal{J}$, let us also define the augmented squares, $\hat{Q}_\al$ and $Q_\al(\eps)$, by
	\begin{align}\label{Q:eps}
		\hat{Q}_\al:=[(i-1)h, (i+2)h]\times[(j-1)h,(j+2)h]\quad\text{and}\quad Q_\al(\eps):=Q_\al+B(0,\eps).
	\end{align}
so that $Q_\al\sub Q_\al(\eps)\sub\hat{Q}_\al$ for each $\al\in\mathcal{J}$, and the ``core," $C_\al(\eps)$, by
	\begin{align}
		C_\al(\eps):=Q_\al(\eps)\smod\bigcup_{\al'\neq\al}Q_{\al'}(\eps)\neq\varnothing,\quad \al\in\mathcal{J}.\notag
	\end{align}
Then there exists a collection of functions $\{\psi_\al\}$ satisfying the properties in Proposition \ref{prop:pou}.  Note that we will use the convention that when $\be$ a positive integer, then $D^\be=\bdy_1^{\be_1}\bdy_2^{\be_2}$, where $\be_1+\be_2=\be$ and $\be_j\geq0$ are integers, while if $\be>0$ is not an integer then $D^\be=\bdy_1^{[\be_1]}\bdy_2^{[\be_2]}\Lam^{\be-[\be]}$, where $[\be]=[\be_1]+[\be_2]$, and finally, if $\be\in(-2,0)$, then $D^\be=\Lam^{\be}$.

\begin{prop}\label{prop:pou}
Let $N\geq9$, $h:=L/\sqrt{N}$, and $\eps:=h/10$.  The collection $\{\psi_{\al}\}_{\al\in\mathcal{J}}$ forms a smooth partition of unity satisfying
	\begin{enumerate}[(i)]
		\item $0\leq\til{\psi}_\al\leq1$ and $\spt\til{\psi}_{\al}\sub ({Q}_\al(\eps)+(2\pi\Z)^2)$;
		\item $\til{\psi}_\al=1$, for all $x\in (C_\al(\eps)+(2\pi\Z)^2)$ and $\sum_{\al\in\mathcal{J}}\til{\psi}_\al(x)=1$, for all $x\in\R^2$;
		\item $c_1h^{2/p}\leq\Sob{\til{\psi}_\al}{L^p(\T^2)}\leq c_2h^{2/p}$, for all $p\in[1,\infty)$, for some constants $c_1, c_2>0$; in particular $(h/(2\pi))^2\leq \lb\til{\psi}_\al\rb\leq c(h/(2\pi))^2$;, for some constant $c>1$;
		\item $\sup_{\al\in\mathcal{J}}\Sob{\til{\psi}_{\al}}{\dot{H}^\be(\T^2)}\ls h^{1-\be}$, for all $\be>-1$;
		\item $\sup_{\al\in\mathcal{J}}\Sob{\til{\psi}_\al}{\dot{H}^{\be}(\T^2)}\ls\left(\frac{2\pi}{h}\right)^{1-\be-\eps(|\be|)} h^{1-\be}$, for all $\be\in(-2,-1]$, for some $\eps\in(1,2)$, where the suppressed constant depends on $\be$;
		\item $\sup_{\al\in\mathcal{J}}\Sob{\Lam^\be D^k\psi_\al}{L^{\infty}(\T^2)}\ls h^{-k-\be}$, for all $\be\in[0,1)$, $k\geq0$ integer.
	\end{enumerate}
\end{prop}

Property (iii) was exploited in \cite{jmt}, but only in the case $p=2$.  We observe here, however, that it also holds for any $p\in[1,\infty)$ since $\ind_{Q_\al}\leq\til{\psi}_\al\leq 1$ and $\spt\til{\psi}_\al\sub(Q_\al(\eps)+(2\pi\Z)^2)$.  On the other hand, property (iv) for $\be\geq0$ was sufficient for the purposes in \cite{jmt}.  We will show here that it also holds $\be\in(-2,0)$, i.e. property $(v)$, as well as the $L^\infty$ estimate $(vi)$.  For this, we will appeal to the following elementary fact:
	\begin{align}\label{lam:rescale}
		(\Lam^{\be}(\phi(\lam\cdotp)))(x)=\lam^\be(\Lam^{\be}\phi)(\lam x),\quad x\in\T^2,\quad \lam>0,
	\end{align}
where we define
	\begin{align}\notag
		\phi(\lam \cdotp)(x):=\phi(\lam x).
	\end{align}
The relation \req{lam:rescale} can be seen easily by appealing to the Fourier transform.  Due to the subtleties of working with periodic functions, we include the details in Lemma \ref{lem:basic:fourier} below.  To this end, let us define
	\begin{align}
		\lb \phi,\psi\rb_{L^2(\Om)}:=\int_{\Om}\phi(x)\overline{\psi(x)}\ dx.\notag
	\end{align}
Let us also denote the Fourier transform on $\T^2$, i.e., for functions which are periodic with period $2\pi$ in $x,y$, by
	\begin{align}
		\mathcal{F}(\phi)(\mathbf{k})=\frac{1}{4\pi^2}\int_{\T^2}e^{-i\mathbf{k}\cdotp x}\phi(x)\ dx,\notag
	\end{align}
and by $\mathcal{F}_\lam$ the Fourier transform on $\lam^{-1}\T^2$, for $\lam>0$, i.e., for functions which are periodic with period $\lam^{-1}2\pi$ in $x,y$, by
	\begin{align}
		(\mathcal{F}_\lam \phi)(\til{\mathbf{k}})=\frac{\lam^2}{4\pi^2}\int_{\lam^{-1}\T^2}e^{-i\til{\mathbf{k}}\cdotp x}\phi(x)\ dx,\quad\til{\mathbf{k}}\in \lam\Z^2.\notag
	\end{align}

\begin{lem}\label{lem:basic:fourier}
Let $\be\in(-2,2]$.  Then
	\begin{enumerate}[(i)]
		\item $\lb \phi,\psi\rb_{L^2(\T^2)}=\lam^2\lb \phi(\lam\cdotp),\psi(\lam\cdotp)\rb_{L^2(\lam^{-1}\T^2)}$, for $\lam>0$.
		\item $\Lam^\be \phi(\lam\cdotp)(x)=\lam^\be(\Lam^\be \phi)(\lam x)$, for $\lam>0$, and any $\be\in\R$, provided that $\phi\in\mathcal{Z}$.
	\end{enumerate}
\end{lem}

\begin{proof}
The first property follows by a change of variables.  Now observe that if $\phi\in C^\infty_{per}(\T^2)\cap\mathcal{Z}$, then $\phi(\lam\cdotp)\in C^\infty_{per}(\lam^{-1}\T^2)\cap\mathcal{Z}$ with period $2\pi\lam^{-1}$ in $x,y$, where $\mathcal{Z}$ is as in \req{Z:space}.  Let $\til{\mathbf{k}}=\lam\mathbf{k}$, for $\mathbf{k}\in\Z^2$.  Then
	\begin{align}
		\mathcal{F}_\lam(\Lam^{\be}\phi(\lam\cdotp))(\til{\mathbf{k}})=\frac{\lam^2}{4\pi^2}\int_{\lam^{-1}\T^2}e^{i\til{\mathbf{k}}\cdotp x}|\til{\mathbf{k}}|^\be \phi(\lam x)\ dx=\lam^\be\frac{1}{4\pi^2}\int_{\T^2}e^{-i\mathbf{k}\cdotp x}|\mathbf{k}|^\be \phi(x)\ dx=\lam^\be\mathcal{F}(\Lam^\be \phi)(\mathbf{k}).\notag
	\end{align}
It follows that for $x\in\lam^{-1}\T^2$, we have
	\begin{align}
		\Lam^\be \phi(\lam\cdotp)(x)=\sum_{\til{\mathbf{k}}\in\lam\Z^2}e^{i\til{\mathbf{k}}\cdotp x}\mathcal{F}_\lam(\Lam^\be \phi(\lam\cdotp)(\til{\mathbf{k}})=\lam^{\be}\sum_{\mathbf{k}\in\Z^2}e^{i\mathbf{k}\cdotp(\lam x)}\mathcal{F}(\Lam^\be \phi)(\mathbf{k})=\lam^{\be}(\Lam^\be \phi)(\lam x).\notag
	\end{align}
\end{proof}

Let us now return to the proof of Proposition \ref{prop:pou} $(v)-(vii)$.  For this, let
	\begin{align}\label{Psi}
		\til{\Psi}_\al(x)=\til{\psi}_\al(hx),
	\end{align}
and $\bar{\Psi}_\al=\til{\Psi}_\al-\left(\frac{h}{2\pi}\right)^2\int_{h^{-1}\T^2}\til{\Psi}_\al(x)\ dx$, so that $\bar{\psi}_\al(x)=\bar{\Psi}_\al(h^{-1}x)$ and $\lb\bar{\Psi}_\al\rb=0$.  Moreover, observe that $\til{\Psi}_\al$ is supported in a square of area $\ls1$.

\begin{proof}[Proof of Proposition \ref{prop:pou} (iv) through (vi)]
\hspace{1pt}

{\flushleft\textit{Proof of (iv) for $\be\in(-1,0)$.}}  For convenience, let $\be>0$.
By Lemma \ref{lem:basic:fourier} (ii), we have
	\begin{align}
		\Sob{\til{\psi}_\al}{\dot{H}^{-\be}(\T^2)}=\Sob{\Lam^{-\be}(\til{\Psi}_\al(h^{-1}\cdotp))}{L^2(\T^2)}=h^{\be}\Sob{(\Lam^{-\be}\til{\Psi}_\al)(h^{-1}\cdotp)}{L^2(\T^2)}=h^{\be+1}\Sob{\Lam^{-\be}\til{\Psi}_\al}{L^2(h^{-1}\T^2)}.\label{be:le1}
	\end{align}
It follows from the Hardy-Littlewood-Sobolev inequality that
	\begin{align}
		\Sob{\Lam^{-\be}\til{\Psi}_\al}{L^2(h^{-1}\T^2)}\leq C\Sob{\til{\Psi}_\al}{L^{2/(1+\be)}(h^{-1}\T^2)}\leq C.\notag
	\end{align}
We see now that from \req{be:le1} we have
	\begin{align}\label{case:be:le1}
		\Sob{\til{\psi}_\al}{\dot{H}^{-\be}(\T^2)}\leq Ch^{\be+1},\quad \be\in(0,1),\notag
	\end{align}
with constant independent of $\alpha$ and $h$, as desired.

{\flushleft\textit{Proof of $(v)$}.}  Let $\be\in[1,2)$.  We estimate by duality.  Indeed, let $\chi\in\dot{H}^\be(\T^2)$ such that $\Sob{\chi}{\dot{H}^\be(\T^2)}$. Then since $\chi\in\mathcal{Z}$, by Parseval's theorem we have
	\begin{align}\notag
		\lb \til{\psi}_\al,\chi\rb_{L^2(\T^2)}=\lb\bar{\psi}_\al,\chi\rb_{L^2(\T^2)}=\lb\Lam^{-\be}\bar{\psi}_\al,\Lam^{\be}\chi\rb_{L^2(\T^2)}.
	\end{align}
Let $q>2/(2-\be)$, so that $q\in(2,\infty)$, and let $q^*\in(1,2)$ be its Sobolev conjugate, i.e., $1/q=1/q^*-\be/2$.  Let $\eps=2/q^*$ and $q'$ denote the H\"older conjugate of $q$.  Observe that $1<q'<2<q<\infty$.  Then by H\"older's inequality, \req{Psi}, and the Hardy-Littlewood-Sobolev inequality, we have
	\begin{align}
		|\lb\til{\psi}_\al,\chi\rb_{L^2(\T^2)}|&\leq\Sob{\Lam^{-\be}\bar{\psi}_\al}{L^{q}(\T^2)}\Sob{\Lam^\be\chi}{L^{q'}(\T^2)}\notag\\
			&\leq (2\pi)^{2/q'-1}h^{\be+2/q}\Sob{\Lam^{-\be}\bar{\Psi}_\al}{L^{q}(h^{-1}\T^2)}\Sob{\Lam^\be\chi}{L^2(\T^2)}\notag\\
			&\leq C\left(\frac{2\pi}{h}\right)^{1-2/q}h^{1+\be}\Sob{\bar{\Psi}_\al}{L^{q^*}(\T^2)}\Sob{\chi}{\dot{H}^\be(\T^2)},\notag\\
			&\leq C\left(\frac{2\pi}{h}\right)^{1+\be-\eps(\be)}h^{1+\be}\notag
	\end{align}
where for the last inequality, we made use of the fact that $|\til{\Psi}_\al|\ls1$ in $h^{-1}\T^2$ and $\til{\Psi}_\al$ is supported in a ball of area $\sim1$.  Thus
	\begin{align}\notag
		\Sob{\til{\psi}_\al}{\dot{H}^{-\be}(\T^2)}\leq C\left(\frac{2\pi}{h}\right)^{1+\be-\eps(\be)}h^{1+\be},\quad\be\in[1,2),
	\end{align}
as desired.

{\flushleft\textit{Proof of $(vi)$}}. The result is trivial when $\be=0$ and $k>0$ simply by rescaling and observing that $D^k\til{\psi}_\al$ is still supported in $Q_\al(\eps)$.  

Suppose that $\be\in(0,1)$.   Now observe that that for $x\in\T^2$, Lemma \ref{lem:basic:fourier} (ii) implies that
	\begin{align}
		(\Lam^\be\til{\psi}_\al)(x)&=c_\be \sum_kp.v.\int_{\T^2}\frac{\til{\psi}_\al(x)-\til{\psi}_\al(y)}{|x-y-2\pi k|^{2+\be}}\ dy\notag\\
			&=\frac{c_\be}{h^{\be}}\sum_kp.v.\int_{h^{-1}\T^2}\frac{\til{\Psi}_\al(h^{-1}x)-\til{\Psi}_\al(y)}{|\frac{x}h-y-\frac{2\pi}{h}k|^{2+\be}}\ dy
=h^{-\be}(\Lam^\be\Psi_\al)(h^{-1}x).
	\end{align}
Since $\Sob{\De\Psi_\al}{L^\infty(h^{-1}\T^2)}\leq C$, this settles the case $\be=2$.  Since $L^\infty$ is invariant under dilations and $\til{\Psi}_\al$ is $2\pi h^{-1}$-periodic in $x,y$, it suffices to consider
	\begin{align}
		\Lam^\be\Psi_\al(x)&=c_\be\sum_kp.v.\int_{h^{-1}\T^2}\frac{\til{\Psi}_\al(x)-\til{\Psi}_\al(y)}{|x-y-\frac{2\pi}{h}k|^{2+\be}}\ dy\notag\\
		&=c_\be p.v.\int_{\R^2}\frac{\til{\Psi}_\al(x)-\til{\Psi}_\al(y)}{|x-y|^{2+\be}}\ dy, \quad x\in h^{-1}\T^2.\notag
	\end{align}
Let us consider two cases: $x\notin 2h^{-1}Q_\al$ and $x\in 2h^{-1}Q_\al$.

If $x\notin 2h^{-1}Q_\al\cap h^{-1}\T^2$, then $\til{\Psi}_\al(x)=0$ and $|x-y|\geq2$.  Thus
	\begin{align}
		|\Lam^\be\Psi_\al(x)|\leq C\Sob{\til{\Psi}}{L^\infty}\int_{|y|\geq2}\frac{dy}{|y|^{2+\be}}\leq C\notag.
	\end{align}
If $x\in 2h^{-1}Q_\al\cap h^{-1}\T^2$, then $|x-y|\leq2$ and we have
	\begin{align}
	\left(
			\int_{\substack{|x-y|\leq 2\\ |y|\leq 1}}
			+\int_{\substack{|x-y|\leq2\\ |y|\geq1}}
	\right)&
	\frac{|\til{\Psi}_\al(x)-\til{\Psi}_\al(y)|}{|x-y|^{2+\be}}\ dy
\notag\\&
	\leq C\Sob{\del\Psi_\al}{L^\infty}\int_{|y|\leq 1}
			\frac{dy}{|y|^{1+\be}}+C\Sob{\Psi_\al}{L^\infty}
			\int_{|y|\geq\de}\frac{dy}{|y|^{2+\be}}
\notag\\&
	\leq \frac{C}{1-\be}\Sob{\del\Psi_\al}{L^\infty}
		+\frac{C}{\be}\Sob{\Psi_\al}{L^\infty}\leq C.\notag
	\end{align}
Thus $|\Lam^\be\Psi_\al(x)|\leq C$ for all $x\in h^{-1}\T^2$, which implies $|\Lam^\be\psi_\al(x)|\leq Ch^{-\be}$ for all $x\in\T^2$, where $C$ is independent of $\al\in\mathcal{J}$.  This establishes $(v)$.
\end{proof}

To ultimately prove \req{TI:3}, {\eric \req{TI:4} and} \req{TI:5}, we will exploit an additional property of the bump functions $\til{\psi}_\al$.  For this, we will make use of the following short-hand for $\phi$ localized to the
squares $Q_\al(\eps)$:
	\begin{align}\notag
		\phi_\al(x)=\phi(x)
			\ind_{Q_{\al}(\eps)}(x),\quad x\in\T^2.
	\end{align}

\begin{lem}\label{lem:loc}
Let $\be\in(-\infty,0)$ and $\phi\in\dot{H}^\be(\T^2)$.
Then there exists a constant $C>0$ such that
	\begin{align}\label{lem:loc:est}
		|\lb\phi,\til{\psi}_\al\rb_{L^2(\T^2)}|\leq Ch^{1+\be}\Sob{\phi_\al}{\dot{H}^\be(\T^2)}+C\left(\frac{h}{2\pi}\right)^2h\Sob{\phi}{L^2(\T^2)}.
	\end{align}
\end{lem}

\begin{proof}

Suppose that $\be\in(-\infty,0)$.  Observe that
	\begin{align}
		\lb\phi,\til{\psi}_\al\rb_{L^2(\T^2)}=\lb\phi_\al,\til{\psi}_\al\rb_{L^2(\T^2)}=\lb{\phi}_\al,\bar{\psi}_\al\rb_{L^2(\T^2)}+\frac{\til{a}(Q_\al)}{4\pi^2}\int_{\T^2}\phi_\al(x)\ dx.\notag
	\end{align}
Then by Parseval's theorem, the Cauchy-Schwarz inequality, and Proposition \ref{prop:pou} (iv), we have
	\begin{align}\label{ip:beneg2}
		|\lb{\phi},{\til{\psi}}_\al\rb|&\leq \Sob{{\phi}_\al}{\dot{H}^\be(\T^2)}\Sob{{\til{\psi}}_\al}{\dot{H}^{|\be|}(\T^2)}+C\left(\frac{h}{2\pi}\right)^2h\Sob{\phi}{L^2(\T^2)}\notag\\
			&\leq Ch^{1-|\be|}\Sob{{\phi}_\al}{\dot{H}^{\be}(\T^2)}+C\left(\frac{h}{2\pi}\right)^2h\Sob{\phi}{L^2(\T^2)},
	\end{align}
as desired.

\end{proof}


%

\subsection{Boundedness properties of volume element interpolants}

For $\phi\in L^1_{loc}(\Om)$, define
	\begin{align}\notag
		\phi_{Q}=\frac{1}{a({Q})}\int_{Q}\phi(x)\ dx\quad\text{and}\quad\til{\phi}_{Q_\al}=\frac{1}{\til{a}({{Q}_\al})}\int_{\T^2}\phi(x)\til{\psi}_\al(x)\ dx,
	\end{align}
where $a({Q})$ denotes the area of $Q$ and
	\begin{align}\label{mod:avg}
		\til{a}({{Q}_\al}):=\int_{\T^2}\til{\psi}_\al(x)\ dx.
	\end{align}
Observe that for each $\al\in\mathcal{J}$, there exists a constant $c>0$, independent of $h, \al,\eps$, such that
	\begin{align}\label{equiv:cubes}
		c^{-1}\leq \frac{\til{a}{({Q_\al})}}{{a}({Q})},  \frac{a{({Q})}}{{a}({\hat{Q}_\al})},  \frac{a{({Q})}}{{a}({Q_\al(\eps)})}\leq c,\quad Q\in\{Q_\al, Q_\al(\eps), \hat{Q}_\al\}.
	\end{align}

We define the smooth volume element interpolant by
	\begin{align}\label{vol:elts}
		\mathcal{I}_h(\phi):=\sum_{\al\in\mathcal{J}}\til{\phi}_{Q_\al}\til{\psi}_\al,
	\end{align}
and the ``shifted" smooth volume element interpolant by
	\begin{align}\label{vol:elts:shift}
		{I}_h(\phi):=\sum_{\al\in\mathcal{J}}\til{\phi}_{Q_\al}\bar{\psi}_\al,\quad \bar{\psi}_\al=\til{\psi}_\al-\lb\til{\psi}_\al\rb.
	\end{align}
	
We will make use of the following elementary fact for a ``square-type" function.  Let $\mathcal{A}$ be a finite index set and $\{A_\al\}_{\al\in\mathcal{A}}\sub\T^2$  be a countable collection of sets such that for each $x\in\T^2$, $\sup_{x\in\T^2}\#\{\al\in\mathcal{A}:x\in Q_\al\}<\infty$.  Define
	\begin{align}\notag
		(S\phi)(x):=\left(\sum_{\al\in\mathcal{A}}(\phi_\al(x))^2\right)^{1/2},\quad \phi_\al(x):=\phi(x)\ind_{A_\al}(x).
	\end{align}
	
\begin{lem}\label{lem:sf}
Let $\phi\in L^1(\T^2)$.  There exists a constant $C>0$ such that
	\begin{align}\label{sq:pw}
		|S\phi(x)|\leq C|\phi(x)|,\quad a.e.\quad x\in\T^2,
	\end{align}
and
	\begin{align}\label{sq:l2}
		\sum_{\al\in\mathcal{A}}\left(\int\phi_{\al}(x)\ dx\right)^2\leq \left(\int S\phi(x)\ dx\right)^2.
\end{align}
\end{lem}

\begin{proof}
Let $N:=\sup_{x\in\T^2}\#\{\al\in\mathcal{A}:x\in A_\al\}$.  Since $N<\infty$, we have that for each $x\in\T^2$, there are at most $N$ sets $A_\al$ such that $x\in A_\al$.  It follows that for each $x\in\T^2$, there exists an integer $C(x)>0$ such that $C(x)\leq N$.  In particular, we have
	\begin{align}
		\sum_\al|\phi_\al(x)|^2= C(x)|\phi(x)|^2\leq N|\phi(x)|^2.\notag
	\end{align}
On the other hand, by Fubini's theorem, and the Cauchy-Schwarz inequality we have that
	\begin{align}\notag
		\sum_\al\left(\int\phi_\al(x)\ dx\right)^2&=\sum_\al\iint \phi_\al(x)\phi_\al(y)\ dxdy\\
	&\leq\iint (S\phi)(x)(S\phi)(y)\ dxdy=\left(\int (S\phi)(x)\ dx\right)^2.\notag
	\end{align}
This completes the proof.
\end{proof}

We immediately obtain the following corollary.

\begin{coro}\label{coro:kernel}
Let $K\in L^1_{loc}(\R^2)$ such that $K\geq0$. Let $\phi\in L^1(\T^2)$ such that $K*\phi\in L^2(\T^2)$.  	\begin{align}
		\sum_{\al\in\mathcal{A}}\Sob{K*\phi_\al}{L^2}^2\leq C\Sob{K*\phi}{L^2}^2.\notag
	\end{align}
In particular, for $\be\in(-2,0)$, we have
	\begin{align}\label{coro:rp}
		\sum_{\al\in\mathcal{J}}\Sob{\phi_\al}{\dot{H}^\be}^2\leq C\Sob{\phi}{\dot{H}^{\be}}^2,
	\end{align}
where $(\mathcal{A},\{A_\al\})$ is given by $(\mathcal{J},\{Q_\al(\eps)\})$ as in \req{Q:eps}.
\end{coro}
\begin{proof}

Observe that
	\begin{align}
		\Sob{K*\phi_\al}{L^2}^2&=\int(K*\phi_\al)(x)^2\ dx\notag\\
		&=\int\left(\int K(x-y)\phi_\al(y)\ dy\right)^2 dx\notag\\
		&\leq\int\int\int K(x-y)K(x-y')|\phi_\al(y)||\phi_{\al}(y')|\ dydy'dx.\notag
	\end{align}
Therefore, by the non-negativity of $K$, the Cauchy-Schwarz inequality, and \req{sq:pw} of Lemma \ref{lem:sf}, we have
	\begin{align}
	\sum_{\al\in\mathcal{A}}\Sob{K*\phi_\al}{L^2}^2&\leq\int \int \int K(x-y)K(x-y')(S\phi)(y)(S\phi)(y')\ dydy'dx\notag\\
			&\leq C^2\int\int \int K(x-y)K(x-y')\phi(y)\phi(y')\ dydy'dx\notag\\
			&\leq C^2\Sob{K*\phi}{L^2}^2.\notag
	\end{align}
It then follows as a special case that \req{coro:rp} holds.  Indeed, the Riesz potential, $\Lam^{\be}$, $\be\in(-2,0)$, has kernel $K(x)\sim|x|^{-2+\be}$, which is locally integrable and non-negative.
\end{proof}


\begin{prop}\label{prop:Hk}
Let $J_h$ be given by either \req{vol:elts} or \req{vol:elts:shift}.  Given $\al\geq1$, let $\eps(\al)$ be as in Proposition \ref{prop:pou} $(v)$ when $\al\in[1,2)$, and identically $0$ otherwise.  Let $C>0$ and define
\begin{align}\label{typeI:const}
		C_I(\al,h):=\begin{cases} C\left(\frac{2\pi}{h}\right),& \al<1,\\
						C \left(\frac{2\pi}{h}\right)^{2+\al-\eps(\al)},&\al\geq1.
				\end{cases}
	\end{align}
There exists a constant $C>0$, depending on $\al$, such that:
\begin{enumerate}

\item If $(\rho,\be)\in[0,\infty)\times[0,2)$, then
	\begin{align}\label{rho:pos}
		\Sob{J_h\phi}{\dot{H}^\rho(\T^2)}\leq C_I(\be,h)
				h^{\be-\rho}\Sob{\phi}{\dot{H}^{\be}(\T^2)}.
	\end{align}
\item If $(\rho,\be)\in[0,\infty)\times(-2,0]$, then
	\begin{align}\label{be:neg}
		\Sob{J_h\phi}{\dot{H}^\rho (\T^2)}\leq Ch^{-\rho}(h^{\be}\Sob{\phi}{\dot{H}^\be}+\Sob{\phi}{L^2(\T^2)}).
	\end{align}
\item If $(\rho,\be)\in(-2,0)\times(-\infty,0]$, then
		\begin{align}\label{rho:neg}
		\Sob{J_h\phi}{\dot{H}^\rho(\T^2)}\leq C_I(|\rho|,h)
				h^{\be-\rho}\Sob{\phi}{\dot{H}^{\be}(\T^2)}.
	\end{align}
\end{enumerate}
\end{prop}

\begin{proof}
We will prove the lemma for the case $J_h$ given by \req{vol:elts}.  The case when $J_h$ is given by \req{vol:elts:shift} is similar.

Let $(\rho,\be)\in[0,\infty)\times[0,2)$. 
Then by Proposition \ref{prop:pou} (iv) and (v), we have
	\begin{align}
		\Sob{J_h\phi}{\dot{H}^\rho(\T^2)}^2&\leq C\sum_\al\til{\phi}_{Q_\al}^2\Sob{\til{\psi}_\al}{\dot{H}^\rho(\T^2)}^2\notag\\
		&\leq Ch^{-2-2\rho }\sum_{\al}|\lb\phi,\til{\psi}_\al\rb|^2\notag\\
		&\leq \til{C}(\be,h)^2h^{-2-2\rho }h^{2+2\be}\sum_{\al}\Sob{\phi}{\dot{H}^\be(\T^2)}^2\notag\\
		&\leq \til{C}(\be,h)^2h^{-2-2\rho }h^{2+2\be}\left(\frac{2\pi}{h}\right)^2\Sob{\phi}{\dot{H}^\be(\T^2)}^2,\notag
	\end{align}
where the constant $\til{C}$ is defined as
	\begin{align}\label{temp:const1}
		\til{C}(\al,h):=\begin{cases}
					C,&\al<1,\\
					C\left(\frac{2\pi}{h}\right)^{1+|\al|-\eps(\al)},&\al\geq1,
			\end{cases}
	\end{align}
where $\eps(\al)>0$ is chosen according to Proposition \ref{prop:pou} $(v)$ and $C>0$ is some constant, depending on $\al$.

 Hence, by \req{typeI:const} we have
	\begin{align}\label{int:case1}
		\Sob{J_h\phi}{\dot{H}^\rho (\T^2)}
\leq C_I(\be,h)h^{\be-\rho }\Sob{\phi}{\dot{H}^\be(\T^2)},
	\end{align}
where $C_I(\be,h)$ is defined by \req{typeI:const}, as desired.


Next, let $(\rho,\be)\in[0,\infty)\times(-2,0]$.  We estimate as before, except that we apply Lemma \ref{lem:loc} and Corollary \ref{coro:kernel} to obtain
	\begin{align}
		\Sob{J_h\phi}{\dot{H}^\rho (\T^2)}^2
		&\leq Ch^{-2-2\rho }\sum_{\al}|\lb\phi,\til{\psi}_\al\rb|^2\notag\\
		&\leq Ch^{2\be-2\rho}\sum_{\al}\Sob{\phi_\al}{\dot{H}^\be(\T^2)}^2+Ch^{-2\rho}\Sob{\phi}{L^2(\T^2)}^2\notag\\
		&\leq Ch^{2\be-2\rho}\Sob{\phi}{\dot{H}^\be}^2+Ch^{-2\rho}\Sob{\phi}{L^2(\T^2)}^2,\notag
	\end{align}
as desired.  


Finally, let $(\rho,\be)\in(-2,0)\times(-\infty,0]$.  To prove \req{rho:neg}, we proceed by duality.  Let $\Sob{\chi}{\dot{H}^{|\rho|}(\T^2)}=1$.  Since $J_h$ is self-adjoint and $\phi\in\mathcal{Z}$,  it follows from Parseval's theorem and \req{rho:pos} that
	\begin{align}
		|\lb J_h\phi,\chi\rb_{L^2(\T^2)}|
					&\leq \Sob{\phi}{\dot{H}^\be(\T^2)}\Sob{J_h\chi}{\dot{H}^{|\be|}(\T^2)}\notag\\
					&\leq \til{C}(|\rho|,h)h^{|\rho|-|\be|}\left(\frac{2\pi}{h}\right)\Sob{\phi}{\dot{H}^\be(\T^2)}\Sob{\chi}{\dot{H}^{|\rho|}(\T^2)}.\notag
	\end{align}
Thus, we have
	\begin{align}
		\Sob{J_h\phi}{\dot{H}^\rho(\T^2)}\leq C_I(|\rho|,h)h^{\be-\rho}\Sob{\phi}{\dot{H}^{\be}(\T^2)},\notag
	\end{align}
as desired.
\end{proof}

\begin{prop}\label{prop:Jh:bddness}
Let $J_h$ be given by \req{vol:elts} or \req{vol:elts:shift}. Let $C_I(\al,h)$ be defined as in \req{typeI:const}.  Define
	\begin{align}\label{typeItil:const}
		\til{C}_I:=\frac{2\pi}{h}.
	\end{align}
Let $\rho,\be\in\R$. There exists a constant $C>0$, depending only on $\rho,\be$, such that
\begin{enumerate}
\item If $\rho\geq0$ and $\be=\ell$ is an integer, then
	\begin{align}\label{rho:betaint}
		\Sob{J_hD^\ell\phi}{\dot{H}^\rho(\T^2)}\leq Ch^{-(\rho+\ell-\ell')}\Sob{\phi}{\dot{H}^{\ell'}(\T^2)},\quad 0\leq\ell'\leq\ell,
	\end{align}
and
	\begin{align}\label{rho:L1}
		\Sob{J_hD^\ell\phi}{\dot{H}^\rho(\T^2)}\leq Ch^{-1-(\rho+\ell-\ell')}\Sob{D^{\ell'}\phi}{L^1(\T^2)},\quad0\leq\ell'\leq\ell.
	\end{align}
\item If $\rho\in(-2,0)$, $\be\in(-2,\infty)$, and $\be'\in(-\infty,\be]$, then
	\begin{align}\label{rho:betaint:neg}
		\Sob{J_hD^\be\phi}{\dot{H}^\rho(\T^2)}\leq C_I(|\rho|,h)h^{-(\rho+\be-\be')}\Sob{\phi}{\dot{H}^{\be'}(\T^2)},
	\end{align}
On the other hand, if $\be=\ell$ is an integer, then
	\begin{align}\label{rho:L1neg}
		\Sob{J_hD^\ell\phi}{\dot{H}^\rho(\T^2)}\leq C_I(|\rho|,h)h^{-1-\rho-\ell}\Sob{\phi}{L^1(\T^2)}.
	\end{align}
\item For $\rho\geq0$ and $\be\in(0,2)$ a non-integer we have
	\begin{align}\label{rho:betanonint}
		\Sob{J_hD^\be\phi}{\dot{H}^\rho(\T^2)}\leq \til{C}_Ih^{-(\rho+\be-\be')}\Sob{\phi}{\dot{H}^{\be'}(\T^2)},\quad0\leq\be'\leq\be.
	\end{align}
\end{enumerate}
\end{prop}

\begin{proof}
Let $\rho\geq0$.
By integrating by parts, Proposition \ref{prop:pou} (iv), and the Cauchy-Schwarz inequality we have
	\begin{align}
		\Sob{J_hD^\ell\phi}{\dot{H}^\rho(\T^2)}^2&\leq C\sum_\al|\til{(D^\ell\phi)}_{Q_\al}|^2\Sob{\til{\psi}_\al}{\dot{H}^\rho(\T^2)}^2\notag\\
					&\leq Ch^{-2-2\rho}\sum_\al|\lb D^{\ell'}\phi, D^{\ell-\ell'}\til{\psi}_\al\rb_{L^2(\T^2)}|^2\notag\\
					&\leq Ch^{-2\rho-2(\ell-\ell')}\sum_\al\Sob{D^{\ell'}\phi}{L^2(Q_\al(\eps))}^2\notag\\
					&\leq Ch^{-2(\rho+\ell-\ell')}\Sob{D^{\ell'}\phi}{L^2(\T^2)}^2,\notag
	\end{align}
which proves \req{rho:betaint}.

Similarly, estimating as before and applying Proposition \ref{prop:pou} (vi) (instead of (iv)) and H\"older's inequality (instead of Cauchy-Schwarz) we have
	\begin{align}
		\Sob{J_hD^\ell\phi}{\dot{H}^\rho(\T^2)}^2&\leq C\sum_\al|\til{(D^\ell\phi)}_{Q_\al}|^2\Sob{\til{\psi}_\al}{\dot{H}^\rho(\T^2)}^2\notag\\
				&\leq Ch^{-2-2\rho}\sum_\al|\lb D^{\ell'}\phi, D^{\ell-\ell'}\til{\psi}_\al\rb_{L^2(\T^2)}|^2\notag\\
				&\leq Ch^{-2-2\rho-2(\ell-\ell')}\sum_\al\Sob{D^{\ell'}\phi}{L^1(Q_\al(\eps))}^2\notag\\
				&\leq Ch^{-2-2(\rho+\ell-\ell')}\Sob{D^{\ell'}\phi}{L^1(\T^2)}^2.\notag
	\end{align}
Arguing as before, we ultimately arrive at \req{rho:L1}.


For $\rho\in(-2,0)$ and $\be'\in(-\infty,\be]$, we proceed by duality.  Indeed, let $\chi\in\dot{H}^{|\rho|}(\T^2)$ with $\Sob{\chi}{\dot{H}^{|\rho|}(\T^2)}=1$.  Since $J_h$ is self-adjoint, by Parseval's theorem we have
	\begin{align}\notag
		|\lb J_hD^\be\phi,\chi\rb_{L^2(\T^2)}|=|\lb\phi, D^{\be}J_h\chi\rb_{L^2(\T^2)}|.
	\end{align}
Then by Parseval's theorem, the fact that $\phi\in\mathcal{Z}$, the Cauchy-Schwarz inequality, the Poincar\'e inequality, and \req{rho:pos} of Proposition \ref{prop:Hk}, we have
	\begin{align}\notag
		|\lb J_hD^\be\phi,\chi\rb_{L^2(\T^2)}|&\leq C\Sob{\phi}{\dot{H}^{\be'}(\T^2)}\Sob{J_h\chi}{\dot{H}^{\be-\be'}(\T^2)}\notag\\
				&\leq C\Sob{\phi}{\dot{H}^{\be'}(\T^2)}\Sob{J_h\chi}{\dot{H}^{\be-\be'}(\T^2)}\notag\\
				&\leq C_I(|\rho|,h)h^{|\rho|-(\be-\be')}\Sob{\phi}{\dot{H}^{\be'}(\T^2)}\Sob{\chi}{\dot{H}^{|\rho|}(\T^2)},\notag
	\end{align}
which implies \req{rho:betaint:neg}.

On the other hand, to prove \req{rho:L1neg}, let $\be=k$ be an integer.  Since $J_h$ is self-adjoint, upon integrating by parts, then applying H\"older's inequality we obtain
	\begin{align}\notag
		|\lb J_hD^k\phi,\chi\rb_{L^2(\T^2)}|&=|\lb\phi,D^kJ_h\chi\rb_{L^2(\T^2)}|\notag\\
					& \leq C\Sob{\phi}{L^1(\T^2)}\Sob{D^kJ_h\chi}{L^\infty(\T^2)}.\notag
	\end{align}
Observe that
	\begin{align}\notag
		(D^kJ_h\chi)(x)=h^{-k}\sum_{\al}\til{\chi}_{Q_\al}(D^k\til{\Psi}_\al)(h^{-1}x).
	\end{align}
Now recall that $N=\sup_{x\in\T^2}\#\{\al\in\mathcal{J}:x\in Q_\al(\eps)\}<\infty$.  Let $\mathcal{J}'(x):=\{\al\in\mathcal{J}:x\in Q_\al(\eps)\}$.  Then it follows from Parseval's theorem, the Cauchy-Schwarz inequality, and Proposition \ref{prop:pou} (iv) through (vi) that
	\begin{align}
		|D^kJ_h\chi(x)|&\leq Ch^{-2-k}\sum_{\al\in\mathcal{J}'(x)}\Sob{\til{\psi}_\al}{\dot{H}^\rho(\T^2)}\Sob{\chi}{\dot{H}^{|\rho|}(\T^2)}\Sob{D^k\til{\Psi}_\al(h^{-1}\cdotp)}{L^\infty(\T^2)}\notag\\
						&\leq C_I(|\rho|,h)Nh^{-1-k}h^{-\rho}\Sob{\chi}{\dot{H}^\rho(\T^2)}\notag.
	\end{align}
Therefore
	\begin{align}\notag
		|\lb J_hD^k\phi,\chi\rb_{L^2(\T^2)}|\leq C_I(|\rho|,h)h^{-1-k-\rho}\Sob{\phi}{L^1(\T^2)}\Sob{\chi}{\dot{H}^\rho(\T^2)},
	\end{align}
which implies \req{rho:L1neg}, as desired.

Finally, we prove \req{rho:betanonint}.  Let $\be\in(0,2)$ a non-integer.  Then by Proposition \ref{prop:pou} (iv) and (v), integration by parts, the fact that $\Lam$ is self-adjoint, and the Cauchy-Schwarz inequality we have
	\begin{align}
		\Sob{J_hD^\be\phi}{\dot{H}^\rho(\T^2)}^2&\leq C\sum_\al|\til{D^\be\phi}_{Q_\al}|^2\Sob{\til{\psi}_\al}{\dot{H}^\rho(\T^2)}^2\notag\\
			&\leq Ch^{-2-\rho}\sum_{\al}|\lb D^{\be'}\phi,D^{\be-\be'}\til{\psi}_\al\rb_{L^2(\T^2)}|^2\notag\\
			&\leq Ch^{-2\rho-2(\be-\be')}\left(\frac{2\pi}{h}\right)^2\Sob{\phi}{\dot{H}^{\be'}(\T^2)}^2.\notag
	\end{align}
This completes the proof.
\end{proof}

\begin{rmk}\label{rmk:LPproj}
We point out that all of the above boundedness properties for $J_h$ hold also when $J_h$ is given by projection onto finitely many Fourier modes, in specific, when $J_h$ is given by the Littlewood-Paley projection.  The only difference is in the constants.   Indeed, one may notice above that this ``defect" between the spectral projection and the ``volume-elements" projection can be traced to the fact the operator, $\Lam^\be, \be\in(-2,2)$, is a non-local operator; although its input may be compactly supported, the output need not have compact support.  Generally speaking, the projection onto Fourier modes up to
wave-number $\ls 1/h$ satisfies convenient ``orthogonality" properties, as captured by the Bernstein inequalities, that are not enjoyed by projection onto local spatial averages.  The above boundedness properties then follow immediately from this inequality and the fact that differential operators will commute $J_h$ when it is given as this projection.  For this reason, we omit the details, but refer to \cite{jmt}, where the relevant estimates are carried out.
\end{rmk}

\bibliographystyle{plain}

\end{document}